\documentclass[a4paper, 12pt]{article}

\usepackage[english]{babel}    

\usepackage[utf8]{inputenc}   

\usepackage[centertags]{amsmath}
\usepackage{amssymb,amsfonts,amsthm}
\usepackage{indentfirst}
\usepackage{array}
\usepackage{float}
\usepackage[usenames,dvipsnames,svgnames,table]{xcolor}
\usepackage{hyperref}
\usepackage{enumerate}
\usepackage{graphicx}
\usepackage{epstopdf}
\usepackage{graphics}
\usepackage{subfigure}
\usepackage[margin=1.5cm,footskip=1.5cm,includefoot]{geometry} 
\usepackage{listings}
\usepackage{setspace}
\usepackage{multirow}
\usepackage{lipsum}

\usepackage{amssymb,amsmath,amsthm,amsfonts,amscd}
\usepackage{graphicx}
\usepackage{color}
\usepackage[all]{xy}
\usepackage{hyperref}
\hypersetup{
    colorlinks=true,
    linkcolor=blue,
    citecolor=blue,
    filecolor=blue,
    urlcolor=blue
}

\newtheorem{mypro}{Proposition}[section]
\newtheorem{myth}[mypro]{Theorem}
\newtheorem{mylem}[mypro]{Lemma}
\newtheorem{mydef}[mypro]{Definition}
\newtheorem{myrem}[mypro]{Remark}
\newtheorem{mycor}[mypro]{Corollary}



\newcommand{\R}{\mathbb{R}}
\newcommand{\N}{\mathbb{N}}
\newcommand{\C}{\mathbb{C}}

\newcommand{\norm}[2]{\|
\hspace{.01in}#1\hspace{.01in}\|
_{#2}}

\DeclareMathOperator*{\esssup}{ess\,sup}
\begin{document}
	%
	%
	%
	%
	\begin{center}
		\sc\large\textbf{}\\
		\vspace{-1em}
		\noindent \hrulefill\\	
		\vspace{-1.1em}
		\noindent \hrulefill\\
	\end{center}
	%
	%
	%
	%
	\begin{center}
		\Large{Estimates for $n$-widths of sets of smooth functions on complex spheres}
	\end{center}
	%
	%
	%
	%
	%
	\noindent {\small 
	Deimer J. J. Aleans\footnote{\texttt{ra162238@ime.unicamp.br}},
	%
	Sergio A. Tozoni\footnote{\texttt{tozoni@ime.unicamp.br\\ The first author was financially supported by Conselho Nacional de Desenvolvimento Científico e Tecnológico (CNPq, Brazil) and  by the Coordenação de Aperfeiçoamento de Pessoal de Nível Superior (CAPES, Brazil).}}	\\
	%
	{\footnotesize
	$^{1}$Instituto de Matem\'atica, Universidade Estadual de Campinas, SP, Brasil\\ 
	%
	$^{2}$Instituto de Matem\'atica, Universidade Estadual de Campinas, SP, Brasil\\
	%
	%
	%
	%
	%
	\begin{abstract}
	In this work we investigate $n$-widths of multiplier operators $\Lambda_*$ and $\Lambda$, defined for functions on the complex sphere 
	$\Omega_d$ of $\C^d$, associated with sequences of multipliers of the type $\{\lambda_{m,n}^*\}_{m,n\in \mathbb{N}}$, $\lambda_{m,n}^*=\lambda(m+n)$ and $\{\lambda_{m,n}\}_{m,n\in \mathbb{N}}$, $\lambda_{m,n}=\lambda(\max\{m,n\})$, respectively,  for a bounded function $\lambda$ defined on $[0,\infty)$. If the operators $\Lambda_{*}$ and $\Lambda$ are bounded from $L^p(\Omega_d)$ into $L^q(\Omega_d)$, $1\leq p,q\leq\infty$, and $U_p$ is the closed unit ball of $L^p(\Omega_d)$, we study lower and upper estimates for the $n$-widths of Kolmogorov, linear, of Gelfand and of Bernstein, of the sets $\Lambda_{*}U_p$ and $\Lambda U_p$ in $L^q(\Omega_d)$. As application we obtain, in particular, estimates for the Kolmogorov $n$-width of classes of Sobolev, of finitely differentiable, infinitely differentiable and analytic functions on the complex sphere, in $L^q(\Omega_d)$, which are order sharp in various important situations.\\

	\noindent {\bf keywords}: Complex sphere, Width, Multiplier, Smooth function.\\
	{\bf MSC 2010 classification numbers}: 41A46, 42B25, 42B15.\\
	\end{abstract}
	\section{Introduction}   
	 In \cite{TozoniI,TozoniII,KushpelI,KushpelII,KushpelIII,KushpelIV,KushpelV,Cuminato,Kushpel&Tozoni,Kushpel&Tozoni8,Tozoni} techniques were developed to obtain estimates for $n$-widths of multiplier
operators defined for functions on torus and  on two-points homogeneous spaces $\mathbb{M}^d$: $\mathbb{S}^d$, $\mathbb{P}^d(\R)$, $\mathbb{P}^d(\C)$, $\mathbb{P}^d(\mathbb{H})$, $\mathbb{P}^{16}(Cay)$. In this work we continue the development of methods of estimating $n$-widths of multiplier operators.

\indent The studies on asymptotic estimates for Kolmogorov $n$-widths of Sobolev classes on the circle were performed
by several important mathematicians such as Rudin, Stechkin, Gluskin, Ismagilov, Maiorov, Makovoz and Scholz
in \cite{Ismagilov, Maiorov, Makovoz, RudinI, Scholz, Stechkin},  they were initiated by Kolmogorov \cite{Kolmogorov} in 1936 and completed by Kashin \cite{Kashin, KashinI} in 1977. Several techniques were applied in these different cases, among them we highlight a technique of
discretization due to Maiorov and the Borsuk theorem. Observing the historical evolution of the study of $n$-widths,
it is possible to note that it has been an usual practice to use different techniques in proofs of lower and upper
bounds and in estimates for classes of finitely and infinitely differentiable functions (see, e.g. \cite{Pinkus}). Among the mathematicians who worked with n-widths of sets of analytic functions, we highlight Babenko, Fisher, Micchelli, Taikov and Tichomirov, who obtained estimates for sets of analytic functions in Hardy spaces in \cite{Babenko, Fisher, Taikov, Tichomirov}. Among the different tools used, we
can emphasize the properties of the class of Blaschke products of degree $m$ or less. In 1980, Melkman in \cite{Melkman} obtained estimates for a set of entire functions in $C[-T, T ]$, using other technique. One of the goals of this work, is to give an unified treatment in the study of $n$-widths of sets of functions defined on complex spheres, determined by multiplier operators. 

Let $\lambda:[0,\infty)\rightarrow\C$ be a bounded function and consider the sequences $\Lambda_{*}=(\lambda_{m,n}^{*})_{m,n\in\N}$, $\lambda_{m,n}^{*}=\lambda(m+n)$ and $\Lambda=(\lambda_{m,n})_{m,n\in\N}$, $\lambda_{m,n}=\lambda(\max\{m,n\})$. Let $\Omega_d$ be the unit sphere of $\C^d$ and for $m,n\in\N$, let $\mathcal{H}_{m,n}$ be the subspace of $L^2(\Omega_d)$ of all complex spherical harmonics of bi-degree $(m,n)$. Given $f\in L^2(\Omega_d)$, for each $m,n\in\N$, there is an unique $Y^{(m,n)}\in\mathcal{H}_{m,n}$ such that $f=\sum_{m,n}Y^{(m,n)}$ with convergence in the norm of $L^2(\Omega_d)$ (see Section 2 for definitions and properties). For $f\in L^2(\Omega_d)$ we define $\Lambda_*f=\sum_{m,n}\lambda_{m,n}^*Y^{(m,n)}$ and $\Lambda f=\sum_{m,n}\lambda_{m,n}Y^{(m,n)}$ where the convergence occurs in the norm of $L^2(\Omega_d)$. Let $U_p$ be the closed unit ball of $L^p(\Omega_d)$, $1\leq p\leq\infty$. In this paper we study asymptotic estimates for the $n$-widths of Kolmogorov, linear, of Gelfand and of Bernstein, of the sets $\Lambda_*U_p$ and $\Lambda U_p$ in $L^q(\Omega_d)$, for $1\leq p,q\leq\infty$.\\
\indent Consider two Banach spaces X and Y. The norm of X will be denoted by $\norm{\cdot}{}$ or $\norm{\cdot}{X}$
 and the closed unit ball $\{x \in X : \norm{x}{}\leq 1\}$ by $B_X$. Let A be a compact, convex and centrally symmetric subset of X. The Kolmogorov $n$-width of A in X is 
 defined by
 \begin{equation*}
  d_n(A,X):=\displaystyle\inf_{ X_{n}} \displaystyle \sup_{x \in A} \hspace{0.1cm} \displaystyle \inf_{y\in X_{n}}  \norm{x-y}{X},
 \end{equation*}
 where $X_{n}$ runs over all subspaces of X of dimension $n$. The linear $n$-width $\delta_n(A,X)$ is defined by
 \begin{equation*}
  \delta_n(A,X):=\displaystyle\inf_{ P_{n}} \displaystyle \sup_{x \in A}\norm{x-P_n(x)}{X}, 
 \end{equation*}
 where $P_n$ runs over all bounded linear operators $P_n:X\longrightarrow X$ whose range has dimension $n$. The Gelfand
 $n$-width of A in X is defined by
 \begin{equation*}
   d^n(A,X)= \displaystyle \inf_{ L^n} \hspace{0.2cm}\displaystyle \sup_{x \in A\cap L^n}  \norm{x}{X},
 \end{equation*}
 where $L^n$ runs over  all subspaces of X of codimension $n$. The Bernstein $n$-width of A in X is defined by
 \begin{equation*} 
 b_{n}(A,X)= \displaystyle \sup_{ X_{n+1}} \hspace{0.2cm}\displaystyle \sup\{\lambda:\lambda B_X \cap X_ {n+1}\subseteq A\}, 
 \end{equation*}
 where $X_{n+1}$ is any $(n+1)$-dimensional subspace of X. The following inequality is always valid:
\begin{equation*}
 b_{n}(A,X)\leq\min\{d_n(A,X),d^n(A,X)\}.
\end{equation*}

We define the $n$-widths of Kolmogorov, linear, Gelfand and Bernstein of a operator $T\in\mathcal{L}(X,Y)$, respectively, by

\begin{equation*}
 d_n(T)=d_n(T(B_X),Y),\hspace{0.2cm} \delta_n(T)=\delta_n(T(B_X),Y),\hspace{0.2cm} d^n(T)=d^n(T(B_X),Y),\hspace{0.2cm} b_n(T)=b_n(T(B_X),Y).
\end{equation*}
Consider $T\in\mathcal{L}(X,Y)$ and let $T^*$ be its adjoint operator. If $T$ is compact or $Y$ is reflexive, then (see \cite{Pinkus}, p. 34)
\begin{equation}\label{(1.2)}
 d^n(T^*)=d_n(T).
\end{equation}

\indent In Section 2 we make a brief study about real and complex spherical harmonics. In Section 3 we study estimates for $n$-widths of sets of functions on the complex sphere $\Omega_d$, associated with the multiplier operators of the type  $\Lambda_{*}$. We prove that the $n$-widths of the sets $\Lambda_{*}U_p$ in $L^q(\Omega_d)$ can be obtained from $n$-widths of sets of functions associated with multiplier operators on the real sphere $S^{2d-1}$. As application we obtain estimates for the Kolmogorov $n$-width of Sobolev classes and of finitely differentiable, infinitely differentiable and analytic functions on the complex sphere.
\indent In Section 4 and in the following sections we study $n$-widths of sets of functions on the complex sphere $\Omega_d$ associated with the multiplier operators of the type  $\Lambda$. We note that in this case the results can not be transferred from results already obtained to the real sphere $S^{2d-1}$. In Section 4 we study estimates for Levy means of special norms on finite-dimensional Banach spaces which are associated with the multiplier sequences $(\lambda_{m,n})_{m,n\in\N}$.\\
\indent In Section 5 we prove two theorems which provide us with lower and upper estimates for Kolmogorov $n$-widths of the sets $\Lambda U_p$ in $L^q(\Omega_d)$, for general multiplier operators of the type $\Lambda$. The main tool used in these proofs are the results from Section 4.\\
\indent In Section 6 we apply the results of Section 5 in the study of the Kolmogorov $n$-widths of the multiplier operators $\Lambda^{(1)}$ and $\Lambda^{(2)}$ associated with the functions $\lambda^{(1)}(t)=t^{-\gamma}(\ln t)^{-\xi}$ for $t>1$, $\lambda^{(1)}(t)=0$ for $0\leq t\leq1$, and $\lambda^{(2)}(t)=e^{-\gamma t^r}$ for $t\geq 0$, where $\gamma,\xi,r\in\R$, $\gamma, r>0$ and $\xi\geq 0$, respectively. We have that the sets $\Lambda^{(1)}U_p$ and $\Lambda_*^{(1)}U_p$ are sets of finitely differentiable functions on $\Omega_d$, and $\Lambda^{(2)}U_p$ and $\Lambda_*^{(2)}U_p$ are sets of infinitely differentiable functions if $0<r<1$ and of analytic functions if $r\geq 1$ (see \cite{MORIMOTOI,MORIMOTOII}). We prove that estimates for the Kolmogorov $n$-widths of the sets $\Lambda^{(1)}U_p$ and $\Lambda^{(2)}U_p$ in $L^q(\Omega_d)$ are sharp when $2\leq p, q < \infty$.\\
\indent We note that the order of decrease of the $n$-widths $d_n(\Lambda^{(1)}U_p,L^q)$ and $d_n(\Lambda_{\ast}^{(1)}U_p,L^q)$ is the same, but it is different for the $n$-widths $d_n(\Lambda^{(2)}U_p,L^q)$ and $d_n(\Lambda_{\ast}^{(2)}U_p,L^q)$.




\indent In this work there are several universal constants which enter into the estimates. These positive constants 
are mostly denoted by the letters $C, C_1,C_2, \ldots$. We did not carefully distinguish between the different constants, neither did we try to get good estimates for them.
The same letter will be used to denote different universal constants in different parts of the paper. For ease of 
notation we will write $a_n\gg b_n$ for two sequences if $a_n\geq Cb_n$ for $n\in\N$, $a_n\ll b_n$ if $a_n\leq Cb_n$ for $n\in\N$,
and  $a_n\asymp b_n$ if  $a_n\ll b_n$  and $a_n\gg b_n$. Also, we will write $(a)_+=a$ if $a>0$ and $(a)_+=0$ if $a\leq 0$.
\section{Analysis on the spheres}
We write $\N=\{0,1,2,\ldots\}$. For $x,y\in\R^d$ we denote by $\langle x,y\rangle$ the usual inner product $\langle x,y\rangle = x_1y_1+\cdots+x_dy_d$. Let 
 \begin{center}
  $S^{d-1}=\{x\in\R^d : \langle x,x\rangle=1\}$
 \end{center}
         be the unit sphere in $\R^d$. Let $\tau_d$ be the Lebesgue measure on $S^{d-1}$, consider the usual Banach spaces $L^p(S^{d-1})$, $1\leq p\leq\infty$ of $\tau_d$-measurable complex functions on $S^{d-1}$ and let $\widetilde{U_p}=\widetilde{U_p}(S^{d-1})=\{\phi \in L^p(S^{d-1}):\norm{\phi}{p}\leq1 \}$. We denote by $\mathcal{H}_k$ the subspace of $L^2(S^{d-1})$ of all spherical harmonics of degree $k$, that is, the space of the restrictions to $S^{d-1}$ of all polynomials $P(x_1,\ldots,x_d)$ which are homogeneous of degree $k$ and harmonic. Given $f\in L^2(S^{d-1})$, for each $k\in\N$, there exists an unique $Y^{(k)}\in \mathcal{H}_k$ such that
 \begin{equation*}
     f=\displaystyle\sum_{k=0}^\infty Y^{(k)},
 \end{equation*}
 where the series converges to $f$ in $L^2(S^{d-1})$. Let $\pi_k:L^2(S^{d-1})\longrightarrow\mathcal{H}_k$ be the harmonic projection operator, $\pi_k(f)=Y^{(k)}$. For each $x\in S^{d-1}$ and $k\in \N$, there exists an unique function $Z^{(k)}_x\in \mathcal{H}_k$, called the zonal harmonic of degree $k$ with pole $x$, such that
 \begin{equation*}
     \pi_k f(x)=\displaystyle\int_{S^{d-1}}f(y)Z_x^{(k)}(y)d_{\tau_d}(y),\hspace{0.5cm} f\in L^2(S^{d-1}).
 \end{equation*}
 We have that $\mathcal{H}_{k}\perp \mathcal{H}_{l}$, for $k\neq l$, with respect to the usual inner product in $L^2(S^{d-1})$ and for $k\geq 2$,
 \begin{equation*}
 d_{k}=\dim \mathcal{H}_{k}= \binom{d+k-1}{k}-\binom{d+k-3}{k-2}.
\end{equation*}
Let $\{Y_j\}_{j=1}^{d_{k}}$ be an orthonormal basis 
 of $\mathcal{H}_{k}$. The following addition formula is known:
 \begin{equation*}
\displaystyle\sum_{j=1}^{d_{k}}\overline{Y_j^{(k)}(x)}Y_j^{(k)}(y)= Z^{(k)}_x(y)=\frac{d+2k-2}{\widetilde{\omega_d}(d-2)} P^{(d-2)/2}_{k}(\langle x,y \rangle),\hspace{1cm}x,y \in S^{d-1},
\end{equation*}
where $\widetilde{\omega_d}$ denotes the surface area of  $S^{d-1}$ and $P^{(d-2)/2}_{k}$
denotes the Gegenbauer polynomial of degree $k$. Let us denote by $\widetilde{Z}^{(k)}_d$  the function 
  \begin{equation*}
  \widetilde{Z}^{(k)}(t)=\widetilde{Z}^{(k)}_d(t)=\frac{d+2k-2}{\widetilde{\omega_d}(d-2)} P^{(d-2)/2}_{k}( t ),\hspace{0.4cm} t\in [-1,1].
 \end{equation*}
 
 Let $\overline{\Upsilon}=\overline{\Upsilon}_d:\C^d\rightarrow \R^{2d}$ be given by 
\begin{equation*}
 \overline{\Upsilon}(z)=(x_1,y_1,\ldots,x_d,y_d),\hspace{0.5cm} z_k=x_k+iy_k,\hspace{0.5cm} k=1,\ldots,d.
\end{equation*}
For $z,w\in\C^d$, we denote by $\langle z,w\rangle$ the usual inner product $\langle z,w\rangle=z_1\overline{w}_1+\cdots+z_d\overline{w}_d$. Let 
 \begin{center}
  $\Omega_d=\{z\in\C^d : \langle z,z\rangle=1\}$
 \end{center}
 be the unit sphere in $\C^d$. We denote by $\Upsilon$ or $\Upsilon_d$, the restriction of $\overline{\Upsilon}$ to 
$\Omega_d$. Since $\Upsilon(\Omega_d)=S^{2d-1}$ we can define the measure $\sigma_d$ on the 
Borel $\sigma$-algebra $\mathcal{B}(\Omega_d)$ of $\Omega_d$ by $\sigma_d(E)=\tau_{2d}(\Upsilon(E))$, $E\in\mathcal{B}(\Omega_d)$. A measurable function $f:\Omega_d\longrightarrow \C$ is $\sigma_d$-integrable if
\begin{equation*}
 \displaystyle\int_{\Omega_d}fd\sigma_d=\displaystyle\int_{S^{2d-1}}(f\circ\Upsilon^{-1})d\tau_{2d}
\end{equation*}
 is well defined. We denote by $L^p=L^p(\Omega_d)$, $1\leq p\leq\infty$, the vector space of all measurable functions defined on $\Omega_d$ and with values in $\C$, satisfying 
\begin{equation*}
 \norm{f}{p}:=\left(\int_{\Omega_d} |f(z)|^p d\sigma_d(z) \right)^{1/p}<\infty, \hspace{0.4cm}1\leq p<\infty,
 \end{equation*}
 \begin{equation*}
  \norm{f}{\infty}:=\esssup_{z\in \Omega_d} |f(z)|<\infty.
 \end{equation*} 
  The space  $L^2(\Omega_d)$ is a Hilbert space with the inner product given by 
 \begin{equation*}
 (f,g)=\displaystyle\int_{\Omega_d}f(z)\overline{g(z)}d\sigma_d(z),\hspace{0.4cm}f,g\in L^2(\Omega_d).
\end{equation*}
Let $U_p=U_p(\Omega_d)=\{\phi \in L^p(\Omega_d):\norm{\phi}{p}\leq1 \}$, $\mathbb{D}=\{z\in\C : |z|\leq1\}$ and let $\nu_\alpha$ be the Borel measure on $\mathbb{D}$ given by
\begin{center}
$\nu_\alpha(E)=\displaystyle\int_{E}(1-|z|^2)^\alpha d\mu_1(z), \hspace{1cm}E\in\mathcal{B}(\mathbb{D}).$
\end{center}
where $\alpha>-1$ and $\mu_1\circ \overline{\Upsilon}^{-1}_1$ is the Lebesgue measure on $\R^2$. For $f\in L^1(\Omega_d)$ and $K\in L^1(\mathbb{D},\nu_\alpha)$, the convolution function 
$f\ast K$ is defined by 
 \begin{center}
 $(f\ast K)(w)=\displaystyle\int_{\Omega_d}f(z)\overline{K(\langle z,w\rangle)}d\sigma_d(z), \hspace{1,5cm}w\in\Omega_d.$
\end{center}

Let  $\alpha=(\alpha_1,\ldots,\alpha_d)\in\N^d$ and $z=(z_1,\ldots,z_d)\in\C^d$. We write $|\alpha|=\alpha_1+\cdots+\alpha_d$, $z^\alpha=z_1^{\alpha_1}\cdots z_d^{\alpha_d}$ and $\overline{z}=(\overline{z}_1,\ldots,\overline{z}_d)$. We denote by $\mathbb{P}(\C^d)$ the vector space of all polynomials in the independent variables $z$ and $\overline{z}$ . If $p\in\mathbb{P}(\C^d)$ then there are $m,n\in\N$ such that  
\begin{equation*}
 p(z):=p(z,\overline{z})=\displaystyle\sum_{|\alpha|\leq m,|\beta|\leq n} p_{\alpha,\beta}z^\alpha\overline{z}^\beta,\hspace{0.5cm} z\in \C^d,\hspace{0.3cm} p_{\alpha,\beta}\in\C,\hspace{0.3cm} \alpha,\beta\in\N^d.
\end{equation*}
 The vector subspace of $\mathbb{P}(\C^d)$ formed by the polynomials that are homogeneous of degree $m$ in the variable $z$ and of
degree $n$ in the variable $\overline{z}$, will be denoted by $\mathbb{P}_{m,n}(\C^d)$. The subspace of $\mathbb{P}_{m,n}(\C^d)$  of all polynomials
that are in the kernel of the complex laplacian
\begin{equation*}
  \Delta_{(2d)}:=4\displaystyle\sum_{j=1}^d\frac{\partial^2}{\partial z_j\partial \overline{z}_j}
 \end{equation*}
 will be denoted by $\mathbb{H}_{m,n}(\C^d)$.\\
 \indent The space of the restrictions of the polynomials in $\mathbb{P}_{m,n}(\C^d)$ to the complex sphere $\Omega_d$ will be denoted by $\mathcal{P}_{m,n}=\mathcal{P}_{m,n}(\Omega_d)$ and of the restrictions of the polynomials in $\mathbb{H}_{m,n}(\C^d)$ to $\Omega_d$ by $\mathcal{H}_{m,n}=\mathcal{H}_{m,n}(\Omega_d)$. We have that $\mathcal{H}_{m,n}\perp \mathcal{H}_{r,s}$, for $(m,n)\neq(r,s)$, with respect to the usual inner product in $L^2(\Omega_d)$. The elements of $\mathcal{H}_{m,n}$ are called complex spherical harmonics of degree $m$ in $z$ and degree $n$ in $\overline{z}$. Let $m,n\in\N$ and $m+n=k$. Since $\mathcal{H}_{m,n}\subset \{Y\circ\Upsilon:Y\in\mathcal{H}_{k}\}$ and there is an orthonormal basis of $\mathcal{H}_{k}$ consisting of only real functions, then also there is an orthonormal basis of $\mathcal{H}_{m,n}$ consisting of only real functions. The dimension of $\mathcal{H}_{m,n}$ is given by 
 \begin{equation*}
 d_{m,n}=\dim \mathcal{H}_{m,n}= \binom{m+d-1}{m}\binom{n+d-1}{n}-\binom{m+d-2}{m-1}\binom{n+d-2}{n-1}.
\end{equation*}
We observe that there is an absolute constant $C>0$ such that 
\begin{equation}\label{eq:(2.2.2)}
 \frac{(m+n)(mn)^{d-2}}{(d-1)!(d-2)!}\leq d_{m,n}\leq \frac{(m+n)(mn)^{d-2}}{(d-1)!(d-2)!}+C(m+n)m^{d-2}n^{d-3},\hspace{0,4cm}(m,n)\in\N^2.
\end{equation}
It is known (see \cite{Rudin}, Chapter 12) that if $f\in L^2(\Omega_d)$ then there exists an unique decomposition
\begin{equation*}
f= \displaystyle\sum_{m,n\in\N} Y^{(m,n)},     
\end{equation*}
where $Y^{(m,n)}\in\mathcal{H}_{m,n}$ for $m,n\in \N$ and the above series converges to $f$ in $L^2(\Omega_d)$. The projection operator $\pi_{m,n}$ from $L^2(\Omega_d)$ on $\mathcal{H}_{m,n}$ is defined by $\pi_{m,n}f=Y^{(m,n)}$. For every $w\in\Omega_d$ there exists an unique $Z_w^{(m,n)}\in\mathcal{H}_{m,n}$ that satisfies
$(\pi_{m,n}f)(w)=(f,Z_w^{(m,n)})$, $f\in L^2(\Omega_d)$.
 The function $Z_w^{(m,n)}$ is called the zonal harmonic of degree $(m,n)$ with pole $w$.\\
 \indent Next, we introduce some classes of disk polynomials. The disk polynomial of degree $m$ in $z$ and degree $n$ in $\overline{z}$ 
 associated to the integer $d-2$ is the polynomial $R_{m,n}^{d-2}$ given by 
 \begin{equation*}
  R_{m,n}^{d-2}(z)=\left \{ \begin{matrix} \frac{1}{P_n^{(d-2,m-n)}(1)}z^{m-n}P_n^{(d-2,m-n)}(2|z|^2-1), &\mbox{if }m\geq n,
\\ \frac{1}{P_m^{(d-2,n-m)}(1)}\overline{z}^{n-m}P_m^{(d-2,n-m)}(2|z|^2-1), & \mbox{if }m<n  ,\end{matrix}\right. 
 \end{equation*}
 where $P_k^{(d-2,m-n)}$ is the usual Jacobi polynomial of degree $k$ associated to the pair of numbers $(d-2,m-n)$. Let $\{Y_j\}_{j=1}^{d_{m,n}}$ be an orthonormal basis 
 of $\mathcal{H}_{m,n}$. The following addition formula is known: 
 \begin{equation*}
 Z^{(m,n)}_w(z)=\displaystyle\sum_{k=1}^{d_{m,n}}\overline{Y_k(w)}Y_k(z)=\frac{d_{m,n}}{\omega_d} R^{d-2}_{m,n}(\langle z,w \rangle),\hspace{1cm}w,z \in \Omega_d,
\end{equation*}
where $\omega_d=\sigma_d(\Omega_d)=\tau_{2d}(S^{2d-1})$. In particular 
\begin{equation*}
    Z^{(m,n)}_w(w)=\displaystyle\sum_{k=1}^{d_{m,n}}|Y_k(z)|^2=\frac{d_{m,n}}{\omega_d}.
\end{equation*}
 Let us denote by $\widetilde{Z}^{(m,n)}$  the function 
  \begin{equation*}
  \widetilde{Z}^{(m,n)}(t)=\frac{d_{m,n}}{\omega_d}R^{d-2}_{m,n}(t),\hspace{0.4cm} t\in \mathbb{D},
 \end{equation*}
 and let $e=e_d=(0,0,...,0,1)$, $f\in L^2(\Omega_d)$. The convolution product $f\ast Z^{(m,n)}_e$ is given by
\begin{eqnarray*}
 f\ast Z^{(m,n)}_e(w)=f\ast \widetilde{Z}^{(m,n)}(w)&=&\displaystyle\int_{\Omega_d} f(z)\overline{Z^{(m,n)}_w(z)}  d\sigma_d(z)=(\pi_{m,n}f)(w).
\end{eqnarray*}
\indent It follow a relationship between the real zonal harmonics on $S^{2d-1}$ and the complex zonal harmonics on $\Omega_d$, as a consequence, we obtain an 
expression relating polynomials on the disk and Gegenbauer polynomials.
\begin{mypro}
 For any $k\in \N$ and $x\in S^{2d-1}$, we have that
 \begin{equation*}
  Z^{(k)}_x\circ \Upsilon = \displaystyle\sum_{m+n=k} Z_w^{(m,n)},\hspace{0.4cm}w=\Upsilon^{-1}(x).
 \end{equation*}
\end{mypro}
\begin{mycor}\label{cor2.2}
 Given $z,w\in \Omega_d$,
 \begin{eqnarray*}
  \frac{2d+2k-2}{\omega_{d}(2d-2)}P_k^{(2d-2)/2}(\langle \Upsilon(z),\Upsilon(w)\rangle)&=&\displaystyle\sum_{m+n=k}\widetilde{Z}^{(m,n)}(\langle z,w \rangle).
 \end{eqnarray*}
\end{mycor}
We denote $\hat{ Z}^{(k)}_d=\sum_{m+n=k}\widetilde{Z}^{(m,n)}$ and we have from the above corollary that 
$\widetilde{Z}^{(k)}_{2d}(\langle \Upsilon(z),\Upsilon(w)\rangle)=\hat{Z}^{(k)}_d(\langle z,w \rangle)$, $z,w\in \Omega_d $. For a detailed study about real and complex spherical harmonics see \cite{Dai, MENEGATTO, Rudin}.
\section{Multiplier operators of type $\Lambda_*$} 

We  denote by $\mathcal{H}$ the vector space generated by the union $\bigcup_{m,n} \mathcal{H}_{m,n}$. The elements of $\mathcal{H}$ are finite linear combinations of elements 
of the subspaces $\mathcal{H}_{m,n}$ and $\mathcal{H}$  is dense in $L^p(\Omega_d)$ for $1\leq p< \infty$. Also we  denote by $\widetilde{\mathcal{H}}$ the vector space generated by the union $\bigcup_{k} \mathcal{H}_{k}$ which is dense in $L^p(S^{2d-1})$ for $1\leq p< \infty$.\\
\indent Let $\Lambda=\{ \lambda_{m,n}\}_{m,n\in\N}$ be a sequence of complex numbers and $1\leq p,q\leq \infty$. If for any $\varphi \in  L^p(\Omega_d)$ there is a function  
$f=\Lambda\varphi\in  L^q(\Omega_d)$ with formal expansion in spherical harmonic 
\begin{equation*}
   f\sim \displaystyle\sum_{m,n\in\N}\lambda_{m,n}\varphi\ast \widetilde{Z}^{(m,n)} ,
  \end{equation*}
 such that $\norm{\Lambda}{p,q}=\sup \{\norm{\Lambda\varphi}{q}:\varphi \in U_p  \}<\infty$, we say that the multiplier operator $\Lambda$ is bounded from $L^p$ into $L^q$
 with norm $\norm{\Lambda}{p,q}$. Now let $\widetilde{\Lambda}=\{ \lambda_{k}\}_{k\in\N}$ be a sequence of complex numbers and $1\leq p,q\leq \infty$. If for any $\varphi \in  L^p(S^{2d-1})$ there is a function  
$f=\widetilde{\Lambda}\varphi\in  L^q(S^{2d-1})$ with formal expansion in spherical harmonic 
\begin{equation*}
   f\sim \displaystyle\sum_{k\in\N}\lambda_{k}\varphi\ast \widetilde{Z}^{(k)} ,
  \end{equation*}
 such that $\norm{\widetilde{\Lambda}}{p,q}=\sup \{\norm{\widetilde{\Lambda}\varphi}{q}:\varphi \in \widetilde{U}_p  \}<\infty$, we say that the multiplier operator $\widetilde{\Lambda}$ is bounded from $L^p$ into $L^q$
 with norm $\norm{\widetilde{\Lambda}}{p,q}$.
 \begin{mylem}\label{lem2.1}
  Let $\lambda:[0,\infty)\longrightarrow \C$ be a bounded function, $\lambda_{m,n}^{\ast}=\lambda(m+n)$ and $\lambda_k=\lambda(k)$ for $m,n,k \in \N$. Consider the multiplier operators associated with the sequences $\Lambda_{\ast}=\{ \lambda_{m,n}^{\ast}\}_{m,n\in \N}$ and $\widetilde{\Lambda}=\{\lambda_k\}_{k\in\N}$ defined  on $\mathcal{H}$ and on $\widetilde{\mathcal{H}}$ respectively. Let $1\leq p,q\leq\infty$  and suppose that $\Lambda_{\ast}$ is bounded from $L^p$ to $L^q$. Then
  \begin{equation}\label{eq2.6}
  \Lambda_{\ast}\varphi=\left(\widetilde{\Lambda}(\varphi\circ \Upsilon^{-1})\right)\circ\Upsilon, \hspace{0.6cm}\varphi \in L^p(\Omega_d)
 \end{equation}
 and
 \begin{equation}\label{eq2.7}
 \Lambda_{\ast} U_p=(\widetilde{\Lambda}\widetilde{U}_p)\circ\Upsilon =\{\varphi \circ\Upsilon : \varphi\in \widetilde{\Lambda}\widetilde{U}_p\}.
 \end{equation}
 \end{mylem}
 \begin{proof}
  For $\varphi \in L^2(\Omega_d)$ we get $\norm{\Lambda_{\ast}\varphi}{2}\leq\left(\sup_{m,n\in \N}|\lambda_{m,n}^{\ast}|\right)\norm{\varphi}{2} $, then $\Lambda_{\ast}$ is a bounded operator on $L^2(\Omega_d)$. Thus
  \begin{eqnarray*}
 \Lambda_{\ast} \varphi&=&\displaystyle\sum_{k=0}^\infty\displaystyle\sum_{m+n=k}\lambda(m+n)\varphi\ast \widetilde{Z}^{(m,n)}\\
 &=&\displaystyle\sum_{k=0}^\infty\lambda(k)\varphi\ast \left(\displaystyle\sum_{m+n=k}\widetilde{Z}^{(m,n)}\right)\\
 &=&\displaystyle\sum_{k=0}^\infty\lambda(k)\varphi\ast \hat{Z}^{(k)}_d,
\end{eqnarray*}
 where the above series converges in $L^2$. On the other hand, from Corollary \ref{cor2.2}, for $w\in \Omega_d$ we have
 \begin{eqnarray*}
 \varphi\ast \hat{Z}^{(k)}_d(w)&=&\displaystyle\int_{\Omega_d} \varphi(z)\overline{\hat{Z}^{(k)}_d(\langle z,w\rangle)}  d\sigma_d(z)\\
 &=&\displaystyle\int_{\Omega_d} \varphi(z)\overline{Z^{(k)}_{\Upsilon(z)}(\Upsilon(w))}  d\sigma_d(z)\\
 &=&\displaystyle\int_{S^{2d-1}} Z^{(k)}_{\Upsilon(w)}(x)(\varphi\circ \Upsilon^{-1})(x)  d\tau_{2d}(x)\\
 &=&(\widetilde{Z}_{2d}^{(k)}\ast(\varphi\circ \Upsilon^{-1}))(\Upsilon(w)).
\end{eqnarray*}
Thus
\begin{equation*}
\Lambda_{\ast}\varphi=\left(\displaystyle\sum_{k=0}^\infty\lambda(k)\widetilde{Z}_{2d}^{(k)}\ast(\varphi\circ \Upsilon^{-1})\right)\circ \Upsilon
=(\widetilde{\Lambda}(\varphi\circ \Upsilon^{-1}))\circ\Upsilon, \hspace{0.6cm}\varphi \in L^2(\Omega_d).
\end{equation*}
For $2\leq p \leq \infty$, we have that $L^\infty(\Omega_d)\subseteq L^p(\Omega_d) \subseteq L^2(\Omega_d)$, then $\Lambda_{\ast}\varphi=(\widetilde{\Lambda}(\varphi\circ \Upsilon^{-1}))\circ\Upsilon$ for all $\varphi\in L^p(\Omega_d)$ with $2\leq p \leq \infty$, in particular for $\varphi \in \mathcal{H}$. It follows for $1\leq p<\infty$  that 
\begin{eqnarray*}
 \norm{\Lambda_{\ast}}{p,q}&=&\sup \{\norm{\Lambda_{\ast}\varphi}{q}:\varphi \in U_p \cap\mathcal{H} \}\\
 &=&\sup \{\norm{(\widetilde{\Lambda}(\varphi\circ \Upsilon^{-1}))\circ\Upsilon}{q}:\varphi\circ \Upsilon^{-1} \in \widetilde{U}_p\cap\widetilde{\mathcal{H}}  \}\\
 &=&\sup \{\norm{\widetilde{\Lambda}h}{q}:h \in \widetilde{U}_p\cap\widetilde{\mathcal{H}}  \}\\
 &=&\norm{\widetilde{\Lambda}}{p,q}.
\end{eqnarray*}
Thus $\widetilde{\Lambda}$ is a bounded operator from $L^p(S^{2d-1})$ to $L^q(S^{2d-1})$.\\
\indent Now consider $1\leq p\leq 2$. Given $\varphi\in L^p(\Omega_d)$ and for each $n\in\N$, let $\varphi_n\in \mathcal{H}$ be such that $\varphi_n \rightarrow\varphi$ in $L^p(\Omega_d)$. Then $\Lambda_{\ast}\varphi_n\longrightarrow\Lambda_{\ast}\varphi$ in $L^q(\Omega_d)$, $\Lambda_{\ast}\varphi_n=(\widetilde{\Lambda}(\varphi_n\circ \Upsilon^{-1}))\circ\Upsilon$ and $\widetilde{\Lambda}(\varphi_n\circ\Upsilon^{-1})\longrightarrow\widetilde{\Lambda}(\varphi\circ\Upsilon^{-1})$ in $L^q(S^{2d-1})$. Therefore $\widetilde{\Lambda}(\varphi_n\circ\Upsilon^{-1})\circ \Upsilon \longrightarrow\widetilde{\Lambda}(\varphi\circ\Upsilon^{-1})\circ \Upsilon$ in $L^q(\Omega_d)$, then $\Lambda_{\ast}\varphi=\widetilde{\Lambda}(\varphi\circ \Upsilon^{-1}))\circ\Upsilon$ q.t.p. . Thus we proved (\ref{eq2.6}) and the proof of (\ref{eq2.7}) follows easily from (\ref{eq2.6}).
 \end{proof}
\begin{myth}\label{lem2.2}
 Let $\Lambda_{\ast}$ and $\widetilde{\Lambda}$ as in Lemma \ref{lem2.1}. Then
 \begin{equation*}
   d_n(\Lambda_{\ast} U_p(\Omega_d),L^q(\Omega_d))= d_n(\widetilde{\Lambda}\widetilde{U}_p(S^{2d-1}),L^q(S^{2d-1})),
 \end{equation*}
 \begin{equation*}
   d^n(\Lambda_{\ast} U_p(\Omega_d),L^q(\Omega_d))= d^n(\widetilde{\Lambda}\widetilde{U}_p(S^{2d-1}),L^q(S^{2d-1})),
 \end{equation*}
 \begin{equation*}
   \delta_n(\Lambda_{\ast} U_p(\Omega_d),L^q(\Omega_d))= \delta_n(\widetilde{\Lambda}\widetilde{U}_p(S^{2d-1}),L^q(S^{2d-1})),
 \end{equation*}
 \begin{equation*}
 b_n(\Lambda_{\ast} U_p(\Omega_d),L^q(\Omega_d))= b_n(\widetilde{\Lambda}\widetilde{U}_p(S^{2d-1}),L^q(S^{2d-1})).
 \end{equation*}

\end{myth}
\begin{proof}
 By the Lemma \ref{lem2.1}
 \begin{eqnarray*}
 d_n(\Lambda_{\ast} U_p(\Omega_d),L^q(\Omega_d))&=&\displaystyle \inf_{ X_{n}} \hspace{0.2cm}\displaystyle \sup_{f \in \Lambda_{\ast} U_p} \hspace{0.2cm} \displaystyle \inf_{g\in X_{n}}  \norm{f-g}{L^q(\Omega_d)}\\
 &=&\displaystyle \inf_{ X_{n}} \hspace{0.2cm}\displaystyle \sup_{f \in \Lambda_{\ast} U_p} \hspace{0.2cm} \displaystyle \inf_{g\in X_{n}}  \norm{f\circ\Upsilon^{-1}-g\circ\Upsilon^{-1}}{L^q(S^{2d-1})}\\
 &=&\displaystyle \inf_{\widetilde{X}_{n}\circ\Upsilon} \hspace{0.2cm}\displaystyle \sup_{f \in (\widetilde{\Lambda}\widetilde{U}_p)\circ\Upsilon} \hspace{0.2cm} \displaystyle \inf_{g\in \widetilde{X}_{n}\circ\Upsilon}  \norm{f\circ\Upsilon^{-1}-g\circ\Upsilon^{-1}}{L^q(S^{2d-1})}\\
 &=&\displaystyle \inf_{\widetilde{X}_{n}} \hspace{0.2cm}\displaystyle \sup_{f \circ\Upsilon^{-1}\in \widetilde{\Lambda}\widetilde{U}_p} \hspace{0.2cm} \displaystyle \inf_{g \circ\Upsilon^{-1}\in \widetilde{X}_{n}}  \norm{f\circ\Upsilon^{-1}-g\circ\Upsilon^{-1}}{L^q(S^{2d-1})}\\
 &=&\displaystyle \inf_{\widetilde{X}_{n}} \hspace{0.2cm}\displaystyle \sup_{\widetilde{f} \in \widetilde{\Lambda}\widetilde{U}_p} \hspace{0.2cm} \displaystyle \inf_{\widetilde{g} \in \widetilde{X}_{n}}  \norm{\widetilde{f}-\widetilde{g}}{L^q(S^{2d-1})}\\
 &=&d_n(\widetilde{\Lambda}\widetilde{U}_p(S^{2d-1}),L^q(S^{2d-1})),
 \end{eqnarray*}
 where $X_n$ runs over all subspaces of $L^q(\Omega_d)$ of dimension $n$ and $\widetilde{X}_n$ over all subspaces of $L^q(S^{2d-1})$ of dimension $n$;
 \begin{eqnarray*}
d^n(\Lambda_{\ast} U_p(\Omega_d),L^q(\Omega_d))&=&\displaystyle \inf_{ L^{n}} \hspace{0.2cm}\displaystyle \sup_{f \in (\Lambda_{\ast} U_p)\cap L^n}  \norm{f}{L^q(\Omega_d)}\\
&=&\displaystyle \inf_{\widetilde{L}^{n}\circ\Upsilon} \hspace{0.2cm}\displaystyle \sup_{f \in ((\widetilde{\Lambda}\widetilde{U}_p)\circ\Upsilon)\cap(\widetilde{L}^n\circ\Upsilon)}\norm{f\circ\Upsilon^{-1}}{L^q(S^{2d-1})}\\
&=&\displaystyle \inf_{\widetilde{L}^{n}} \hspace{0.2cm}\displaystyle \sup_{f\circ\Upsilon^{-1} \in (\widetilde{\Lambda}\widetilde{U}_p)\cap\widetilde{L}^n}\norm{f\circ\Upsilon^{-1}}{L^q(S^{2d-1})}\\
&=&\displaystyle \inf_{\widetilde{L}^{n}} \hspace{0.2cm}\displaystyle \sup_{\widetilde{f} \in (\widetilde{\Lambda}\widetilde{U}_p)\cap\widetilde{L}^n}\norm{\widetilde{f}}{L^q(S^{2d-1})}\\
 &=&d^n(\widetilde{\Lambda}\widetilde{U}_p(S^{2d-1}),L^q(S^{2d-1})),
 \end{eqnarray*}
 where $L^n$ runs over all subspaces of $L^q(\Omega_d)$ of codimension $n$ and $\widetilde{L}^n$ over all subspaces of $L^q(S^{2d-1})$ of codimension $n$;
 \begin{eqnarray*}
 \delta_n(\Lambda_{\ast} U_p(\Omega_d),L^q(\Omega_d))&=&\displaystyle \inf_{ P_{n}} \hspace{0.2cm}\displaystyle \sup_{f \in \Lambda_{\ast} U_p}  \norm{f-P_{n}(f)}{L^q(\Omega_d)}\\
 &=&\displaystyle \inf_{\widetilde{P}_{n}\circ\Upsilon} \hspace{0.2cm}\displaystyle \sup_{f \in (\widetilde{\Lambda}\widetilde{U}_p)\circ\Upsilon}\norm{f\circ\Upsilon^{-1}-\widetilde{P}_{n}(f\circ\Upsilon^{-1})}{L^q(S^{2d-1})}\\
  &=&\displaystyle \inf_{\widetilde{P}_{n}} \hspace{0.2cm}\displaystyle \sup_{f \circ\Upsilon^{-1}\in \widetilde{\Lambda}\widetilde{U}_p} \norm{f\circ\Upsilon^{-1}-\widetilde{P}_{n}(f\circ\Upsilon^{-1})}{L^q(S^{2d-1})}\\
 &=&\displaystyle \inf_{\widetilde{P}_{n}} \hspace{0.2cm}\displaystyle \sup_{\widetilde{f} \in \widetilde{\Lambda}\widetilde{U}_p} \norm{\widetilde{f}-\widetilde{P}_{n}(\widetilde{f})}{L^q(S^{2d-1})}\\
 &=&\delta_n(\widetilde{\Lambda}\widetilde{U}_p(S^{2d-1}),L^q(S^{2d-1})),
 \end{eqnarray*}
 where $P_n$ runs over all bounded linear operators $P_n:L^q(\Omega_d)\longrightarrow L^q(\Omega_d)$ whose range has dimension $n$ and $\widetilde{P}_n$ over all bounded linear operators $\widetilde{P}_n:L^q(S^{2d-1})\longrightarrow L^q(S^{2d-1})$ whose range has dimension $n$;
 \begin{eqnarray*}
 b_{n}(\Lambda_{\ast} U_p(\Omega_d),L^q(\Omega_d))&=& \displaystyle \sup_{ X_{n+1}} \hspace{0.2cm}\displaystyle \sup\{\lambda:\lambda U_q \cap X_ {n+1}\subseteq \Lambda_{\ast} U_p(\Omega_d) \}\\
 &=&\displaystyle \sup_{ X_{n+1}} \hspace{0.2cm}\displaystyle \sup\{\lambda:\lambda U_q \cap X_ {n+1}\subseteq (\widetilde{\Lambda}\widetilde{U}_p)\circ\Upsilon \}\\
 &=&\displaystyle \sup_{ \widetilde{X}_{n+1}\circ\Upsilon} \hspace{0.2cm}\displaystyle \sup\{\lambda:\lambda (\widetilde{U}_{q}\circ\Upsilon) \cap (\widetilde{X}_{n+1}\circ\Upsilon)\subseteq (\widetilde{\Lambda}\widetilde{U}_p)\circ\Upsilon \}\\
 &=&\displaystyle \sup_{ \widetilde{X}_{n+1}} \hspace{0.2cm}\displaystyle \sup\{\lambda:\lambda \widetilde{U}_{q} \cap \widetilde{X}_{n+1}\subseteq \widetilde{\Lambda}\widetilde{U}_p \}
 =b_n(\widetilde{\Lambda}\widetilde{U}_p(S^{2d-1}),L^q(S^{2d-1})),
 \end{eqnarray*}
 where $X_{n+1}$ runs over all subspaces of $L^q(\Omega_d)$ of dimension $n+1$ and $\widetilde{X}_{n+1}$ over all subspaces of $L^q(S^{2d-1})$ of dimension $n+1$.
 \end{proof}

 \begin{myrem}
  For $\gamma\in \R$, we donote $\Lambda^\gamma=\{\lambda_{m,n}^\gamma\}_{m,n\in \N}$, with $\lambda_{m,n}^\gamma=\lambda^\gamma(m+n)$, where the function 
$\lambda^\gamma:[0,\infty)\longrightarrow \R$ is defined by $\lambda^\gamma(t)=(t(t+2d-2))^{-\gamma/2}$ if $t>0$ and $\lambda(0)=0$. The Sobolev space $\overline{W}_p^\gamma$ on $\Omega_d$ with $\gamma>0$, $1\leq p\leq \infty$, is defined as the set 
  \begin{equation*}
  \overline{W}_p^\gamma=\{\varphi\in L^p(\Omega_d):\Lambda^{-\gamma}\varphi\in L^p(\Omega_d)\}
  =\{c+\Lambda^\gamma\psi:c\in\R,\hspace{0.2cm}\psi\in L^p(\Omega_d)\},
 \end{equation*}
 endowed with norm 
 \begin{equation*}
 \norm{\varphi}{\overline{W}_p^\gamma}=\norm{\Lambda^{-\gamma}\varphi}{p}+\norm{\varphi}{p}.
\end{equation*}
We define the Sobolev classe $W_p^\gamma$ as the set
 \begin{equation*}
  W_p^\gamma=\{\Lambda^{\gamma}\varphi : \varphi \in U_p\}=\Lambda^\gamma U_p,\hspace{0.4cm}\gamma>0,\hspace{0.4cm}1\leq p\leq\infty.
 \end{equation*}
 Now let $\widetilde{\Lambda}^\gamma=\{\lambda^{\gamma}_k\}_{k \in \N}$, $\lambda^{\gamma}_k=\lambda^{\gamma}(k)$. The Sobolev classe $\widetilde{W}_p^\gamma$ on $S^{2d-1}$ is given by 
  \begin{equation*}
   \widetilde{W}_p^\gamma=\{\widetilde{\Lambda}^\gamma \varphi ;\varphi \in \widetilde{U}_p\}=\widetilde{\Lambda}^\gamma\widetilde{U}_p.
 \end{equation*}
 \end{myrem}
\noindent Applying the Lemma \ref{lem2.1} for the operator $\Lambda^{\gamma}$ we get 
 \begin{equation*}
  W_p^\gamma= \{\varphi \circ\Upsilon ; \varphi\in \widetilde{W}_p^\gamma\},
\end{equation*}
and by Theorem \ref{lem2.2} we have that 
\begin{equation}\label{eq:(2.3.4)}
d_n(W_p^\gamma(\Omega_d),L^q(\Omega_d))= d_n(\widetilde{W}_p^\gamma(S^{2d-1}),L^q(S^{2d-1})),
\end{equation}
 \begin{equation}\label{eq:(0010)}
d^n(W_p^\gamma(\Omega_d),L^q(\Omega_d))= d^n(\widetilde{W}_p^\gamma(S^{2d-1}),L^q(S^{2d-1})),
\end{equation}
\begin{equation}\label{eq:(0011)}
\delta_n(W_p^\gamma(\Omega_d),L^q(\Omega_d))= \delta_n(\widetilde{W}_p^\gamma(S^{2d-1}),L^q(S^{2d-1})),
\end{equation}
 \begin{equation}\label{eq:(0012)}
b_n(W_p^\gamma(\Omega_d),L^q(\Omega_d))= b_n(\widetilde{W}_p^\gamma(S^{2d-1}),L^q(S^{2d-1})).
\end{equation}
\indent The estimates in the theorem below  follow as a consequence of (\ref{eq:(2.3.4)}) and of estimates already obtained for the Kolmogorov $n$-width for the Sobolev classes $\widetilde{W}_p^\gamma$ on the real sphere $S^{2d-1}$, which can be found in  \cite{TozoniI, TozoniII, KushpelII, KushpelIII, KushpelV, Kushpel&Tozoni, Tozoni}.

\begin{myth}\label{myth2.1}
  For the Kolmogorov n-width of Sobolev classes on $\Omega_d$, we have in particular the following estimates:
  \begin{itemize}
  \item[(a)] if $1\leq p=q\leq \infty$, $2\leq q\leq p<\infty$, $\gamma>0$,
  \begin{center}
   $d_n(W_p^\gamma(\Omega_d),L^q(\Omega_d))\asymp n^{-\gamma/(2d-1)}$;
  \end{center}
   \item[(b)] if $1\leq p=q\leq \infty$, $2\leq q\leq p\leq\infty$, $\gamma>0$,
   \begin{center}
    $n^{-\gamma/(2d-1)}(\ln n)^{-1/2}\ll d_n(W_p^\gamma(\Omega_d),L^q(\Omega_d))\ll n^{-\gamma/(2d-1)}$;
   \end{center}
   \item[(c)] if $2\leq q\leq p <\infty$, $\gamma/(2d-1)>1/2$,
   \begin{center}
    $d_n(W_p^\gamma(\Omega_d),L^q(\Omega_d))\asymp n^{-\gamma/(2d-1)}$;
   \end{center}
   
   \item[(d)] if $2\leq q\leq p \leq \infty$, $\gamma/(2d-1)>1/2$,
   \begin{center}
    $n^{-\gamma/(2d-1)}\ll d_n(W_p^\gamma(\Omega_d),L^q(\Omega_d))\ll n^{-\gamma/(2d-1)}(\ln n)^{1/2}$;
   \end{center}
   
    \item[(e)] if $1\leq p\leq q \leq 2$, $\gamma/(2d-1)>1/p-1/q$,
    \begin{center}
   $d_n(W_p^\gamma(\Omega_d),L^q(\Omega_d))\asymp n^{-\gamma/(2d-1)+1/p-1/q}$; 
  \end{center}
  
   \item[(f)] if $1\leq p\leq q \leq 2$, $\gamma>0$,
    \begin{center}
   $d_n(W_p^\gamma(\Omega_d),L^q(\Omega_d))\asymp n^{-\gamma/(2d-1)}$; 
  \end{center}
  
  \item[(g)] if $1\leq p\leq q \leq 2$, $\gamma>0$,
    \begin{center}
   $n^{-\gamma/(2d-1)}(\ln n)^{-1/2}\ll d_n(W_p^\gamma(\Omega_d),L^q(\Omega_d))\ll n^{-\gamma/(2d-1)}$;
  \end{center}
  
  \item[(h)] if $1\leq p \leq 2 \leq q < \infty$, $\gamma/(2d-1) > 1/p$,
  \begin{center}
   $d_n(W_p^\gamma(\Omega_d),L^q(\Omega_d))\asymp n^{-\gamma/(2d-1)+1/p-1/2}$;
  \end{center}
   
   \item[(i)] if $1\leq p \leq 2 \leq q \leq \infty$, $\gamma/(2d-1) >1/p$,
  \begin{center}
   $n^{-\gamma/(2d-1)+1/p-1/2}\ll d_n(W_p^\gamma(\Omega_d),L^q(\Omega_d))\ll n^{-\gamma/(2d-1)+1/p-1/2}(\ln n)^{1/2}$.
  \end{center}
  
  \end{itemize}

 \end{myth}
 Other estimates for the Kolmogorov $n$-widths of the Sobolev classes $W_p^\gamma(\Omega_d)$, as also to the $n$-widths of Gelfand, linear and Bernstein, can be obtained using (\ref{eq:(0010)})-(\ref{eq:(0012)}).
 \begin{myrem}
   For $\gamma,\xi\in \R$, $\gamma>0$ and $\xi\geq 0$ , we denote $\Lambda_{\ast}^{(1)}=\{\lambda_{m,n}^{(1),\ast}\}_{m,n\in \N}$, with $\lambda_{m,n}^{(1),\ast}=\lambda^{(1)}(m+n)$, where the function 
$\lambda^{(1)}:[0,\infty)\longrightarrow \R$ is defined by $\lambda^{(1)}(t)=t^{-\gamma}(\ln t)^{-\xi}$ for $t>1$ and $\lambda^{(1)}(t)=0$ for $0\leq t\leq 1$. We have that $\Lambda_{\ast}^{(1)} U_p$ is a set of finitely differentiable functions on $\Omega_d$, in particular, $\Lambda_{\ast}^{(1)} U_p$ is a Sobolev-type class 
if $\xi=0$. Now for $\gamma,r \in \R$, with $\gamma,r>0$, we denote $\Lambda_{\ast}^{(2)}=\{\lambda_{m,n}^{(2),\ast}\}_{m,n\in \N}$ with $\lambda_{m,n}^{(2),\ast}=\lambda^{(2)}(m+n)$, where the function $\lambda^{(2)}:[0,\infty)\longrightarrow \R$ is defined by $\lambda^{(2)}(t)=e^{-\gamma t^r}$ for $t\geq 0$. We have that $\Lambda_{\ast}^{(2)}U_p$ is a set of infinitely differentiable $(0<r<1)$ or analytic $(r\geq 1)$ functions on $\Omega_d$.
 \end{myrem}
 \indent The results of the next three theorems follow as a consequence of the Theorem \ref{lem2.2} and of results proved in \cite{Tozoni}.
 \begin{myth}
  If $1\leq p\leq \infty$, $2\leq q\leq\infty$ and $\gamma/(2d-1) > 1/p$, we have
\begin{center}
 $d_n(\Lambda_{\ast}^{(1)}U_p,L^q)\ll n^{-\gamma/(2d-1)+(1/p-1/2)_+}(\ln n)^{-\xi}\left \{ \begin{matrix} q^{1/2}, &\mbox{if } q<\infty,
\\ (\ln n)^{1/2}, & \mbox{if }q=\infty,\end{matrix}\right.$
\end{center}
\noindent and for $\gamma/(2d-1)> 1/p-1/q$,
 \begin{center}
 $d_n(\Lambda_{\ast}^{(1)}U_p,L^q)\gg n^{-\gamma/(2d-1)}(\ln n)^{-\xi}\vartheta_n$,
\end{center}
where
\begin{center}
 $\vartheta_n=\left \{ \begin{matrix} 1, &\mbox{if }1\leq p\leq 2, 1<q\leq 2,
\\ 1, & \mbox{if }2\leq p<\infty, 2\leq q\leq\infty,
\\ 1,& \mbox{if }1\leq p\leq 2 \leq q\leq\infty,
\\ (\ln n)^{-1/2}, & \mbox{if }1\leq p \leq 2, q=1,
\\ (\ln n)^{-1/2}, & \mbox{if }p=\infty, 2\leq q \leq \infty.\end{matrix}\right.$
\end{center}

\noindent In particular, if $\gamma/(2d-1)> 1/p, 2\leq p,q<\infty$,
\begin{equation*}
 d_n(\Lambda_{\ast}^{(1)}U_p,L^q)\asymp n^{-\gamma/(2d-1)}(\ln n)^{-\xi}.
\end{equation*}
 \end{myth}
\begin{myth}
 For $\gamma,r\in\R$ such that  $\gamma>0$, $0<r\leq1$ and for all  $k\in\N$ we have that 
\begin{equation*}
 d_k(\Lambda_{\ast}^{(2)}U_p,L^q)\gg e^{-\mathcal{R}_{\ast}k^{r/(2d-1)}}\vartheta_k,
\end{equation*}
\noindent where 
 \begin{equation}
  \mathcal{R}_{\ast}=\gamma\ \left(\frac{(2d-1)!}{2}\right) ^{r/(2d-1)},\nonumber
 \end{equation}
and  $\vartheta_k$ is given in the previous theorem.
\end{myth}

\begin{myth}
  For $0<r\leq1$, $1\leq p\leq\infty$, $ 2\leq q\leq\infty$ and for all $k\in\N$, we have that 
 \begin{equation*}
 d_k(\Lambda_{\ast}^{(2)}U_p,L^q)\ll e^{-\mathcal{R}_{\ast}k^{r/(2d-1)}}k^{(1-r/(2d-1))(1/p-1/2)_+}\left \{ \begin{matrix} q^{1/2}, &\mbox{if } q< \infty,\\ 
(\ln k)^{1/2}, & \mbox{if }q=\infty,\end{matrix}\right.
\end{equation*}
where $\mathcal{R}_{\ast}$ is the constant given in the previous theorem.
\end{myth}
\indent Other estimates for the Kolmogorov $n$-width for the operator $\Lambda_{\ast}^{(2)}$ when $r\in\R$, $r>1$, can also be obtained from results in \cite{Tozoni}. 
 
\section{Estimates for Levy means}
In this paper, from this point forward, we will consider only the multiplier operators $\Lambda$ associated with the norm $|(m,n)|=\max \{m,n\}$.\\
\indent Given $l,N,m,n,M_1,M_2 \in \N$, with $M_1<M_2$, we consider the following notations:
\begin{center}
  $A_l=\{(m,n)\in \N^2:|(m,n)|\leq l\},\hspace{0.4cm}A_{-1}=\emptyset\hspace{0.4cm},a_l=\#(A_l\setminus A_{l-1})$  
 \end{center}
 \begin{equation*}
 \mathcal{H}_l=\displaystyle\bigoplus_{(m,n)\in A_l\setminus A_{l-1} }\mathcal{H}_{m,n},\hspace{0.6cm}d_l= \dim \mathcal{H}_{l}=\displaystyle\sum_{(m,n)\in A_l\setminus A_{l-1}}d_{m,n},
\end{equation*}
\begin{equation*}
 \mathcal{T}_{N}=\displaystyle\bigoplus_{l=0}^N\mathcal{H}_l=\displaystyle\bigoplus_{(m,n)\in A_N }\mathcal{H}_{m,n},\hspace{0.6cm}\mathcal{T}_{M_1,M_2}=\displaystyle\bigoplus_{(m,n)\in A_{M_2}\setminus A_{M_1}} \mathcal{H}_{m,n}\hspace{0.6cm}and\hspace{0.6cm}s=\dim\mathcal{T}_{M_1,M_2}.
\end{equation*}
In this section and in the following sections we will consider $\mathcal{H}_{m,n}$ and $\mathcal{T}_{M_1,M_2}$ as real vector subspaces of the real vector spaces $L^p(\Omega_d,\R)$, $1\leq p\leq\infty$.
\begin{mypro}\label{pr:(3.1.1)}
 There are positive constants $C_1,C_2,C_3$ such that
 \begin{equation}\label{eq:(3.1.1)}
  \frac{2(2d-1)}{d!(d-1)!}l^{2d-2}-C_1l^{2d-3}\leq \dim \mathcal{H}_{l}\leq \frac{2(2d-1)}{d!(d-1)!}l^{2d-2}+ C_2l^{2d-3},
 \end{equation}
 \begin{equation}\label{eq:(3.1.2)}
  \frac{2}{d!(d-1)!}N^{2d-1}\leq \dim \mathcal{T}_{N}\leq \frac{2}{d!(d-1)!}N^{2d-1}+C_3 N^{2d-2}.
 \end{equation}
In particular, $\dim \mathcal{H}_{l}\asymp l^{2d-2}$ and $\dim \mathcal{T}_{N}\asymp N^{2d-1}$.
\end{mypro}
 \begin{proof}
  The proofs of (\ref{eq:(3.1.1)}) and (\ref{eq:(3.1.2)}) are easily obtained using (\ref{eq:(2.2.2)}) and integration to estimate finite sums.
 \end{proof}
\indent Let $ E =(\mathbb{R}^n,\|\cdot\|)$ be a $n$-dimensional real Banach space with unit ball $B_E=\{x\in \R^n:\norm{x}{}\leq1\}$, and let $|||x|||=(\sum_{k=1}^{n}|x_k|^2)^{1/2}$ and $\langle x,y\rangle=\sum_{k=1}^{n}x_ky_k$ be the euclidean norm and the inner product on $\R^n$. Let $S^{n-1}=\{x\in \R^n:|||x|||= 1\}$ be the euclidean unit sphere in $\R^n$. The Levy mean of the norm $\norm{\cdot}{}$ on $\R^n$ is defined by 
\begin{equation*}
 M(\norm{\cdot}{})=M(\R^n,\norm{\cdot}{})=\left(\displaystyle \displaystyle\int_{S^{n-1}} \norm{x}{}^2  d\mu(x)\right)^{1/2}.
 \end{equation*}
where $\mu$ denotes the normalized Lebesgue measure on $S^{n-1}$.
\begin{myrem}\label{myrem3.8}
 For $M_{1}+1\leq l\leq M_2$ let $A_l\setminus A_{l-1}=\{(m_j^l,n_j^l):1\leq j\leq a_l\}$ such that $|(m_j^l,n_j^l)|\leq|(m_{j+1}^l,n_{j+1}^l)|$ for $1\leq j\leq a_l-1$ and let $\{Y_1^{(m_j^l,n_j^l)},\ldots,Y_{d_{m_j^l,n_j^l}}^{(m_j^l,n_j^l)}\}$ be an orthonormal basis of $\mathcal{H}_{m_j^l,n_j^l}$ consisting of only real functions. We denote $Y_i^{l,j}=Y_i^{(m_j^l,n_j^l)}$, $d_{l,j}=d_{m_j^l,n_j^l}$ and $\mathcal{H}_{l,j}=\mathcal{H}_{m_j^l,n_j^l}$. We consider the orthonormal basis
\begin{center}
$\{\xi_k\}_{k=1}^s=\displaystyle\bigcup_{l=M_1+1}^{M_2}\displaystyle\bigcup_{j=1}^{a_l}\{Y_1^{l,j},\ldots,Y_{d_{l,j}}^{l,j}\}$
\end{center}
of $\mathcal{T}_{M_1,M_2}$ endowed with the order
$Y_1^{M_1+1,1},\ldots,Y_{d_{M_1+1,1}}^{M_1+1,1},\ldots,Y_1^{M_1+1,a_{M_1+1}},\ldots,Y_{d_{M_1+1,a_{M_1+1}}}^{M_1+1,a_{M_1+1}},\\ \ldots,
 Y_1^{M_2,1},\ldots,Y_{d_{M_2,1}}^{M_2,1},\ldots,
 Y_1^{M_2,a_{M_2}},\ldots,Y_{d_{M_2,a_{M_2}}}^{M_2,a_{M_2}}$.
 Let $J:\R^s\longrightarrow \mathcal{T}_{M_1,M_2}$ be the coordinate isomorphism which assigns to $\alpha=(\alpha_1,\ldots ,\alpha_s)\in\R^s$ the function
 \begin{center}
 $J(\alpha)=\displaystyle\sum_{k=1}^s\alpha_k\xi_k=\displaystyle\sum_{l=M_1+1}^{M_2}\displaystyle\sum_{j=1}^{a_l}\displaystyle\sum_{i=1}^{d_{l,j}}\alpha_i^{l,j}Y_i^{l,j}$
\end{center}
where $(\alpha_1,\ldots ,\alpha_s)=(\alpha_1^{M_1+1,1},\ldots ,\alpha_{d_{M_1+1,1}}^{M_1+1,1},\ldots,\alpha_1^{M_1+1,a_{M_1+1}},\ldots ,\alpha_{d_{M_1+1,a_{M_1+1}}}^{M_1+1,a_{M_1+1}},\ldots,$\\
$\alpha_1^{M_2,1},\ldots,\alpha_{d_{M_2,1}}^{M_2,1},\ldots,\alpha_1^{M_2,a_{M_2}},\ldots,\alpha_{d_{M_2,a_{M_2}}}^{M_2,a_{M_2}})$.\\
\indent Consider a function $\lambda:[0,\infty)\longrightarrow \R$, such that  $\lambda(t)\neq 0$ for $t\geq 0$ and let $\Lambda=\{\lambda_{m,n}\}_{m,n\in \N}$ be the sequence of multipliers defined by $\lambda_{m,n}=\lambda(|(m,n)|)$.
For $M_1+1\leq l\leq M_2$ and $1\leq j\leq a_l$, we write $\lambda_j^l=\lambda_{m_j^l,n_j^l}=\lambda(|(m_j^l,n_j^l)|)$. Let  $\Lambda_s=\{\widetilde{\lambda}_k \}_{k=1}^s$ be the numerical sequence
\begin{equation*}
\Lambda_s =\{\underbrace{\lambda_1^{M_1+1},\ldots,{\lambda_1^{M_1+1}}}_{{d_{M_1+1,1}}},\ldots,\underbrace{\lambda_{a_{M_1+1}}^{M_1+1},\ldots,{\lambda_{a_{M_1+1}}^{M_1+1}}}_{{d_{M_1+1,a_{M_1+1}}}},\ldots,\underbrace{\lambda_1^{M_2},\ldots,{\lambda_1^{M_2}}}_{{d_{M_2,1}}},\ldots,\underbrace{\lambda_{a_{M_2}}^{M_2},\ldots,{\lambda_{a_{M_2}}^{M_2}}}_{{d_{M_2,a_{M_2}}}} \}.
\end{equation*}
Consider the multiplier operator $\Lambda_s$ on $\mathcal{T}_{M_1,M_2}$ defined by
\begin{center}
 $\Lambda_s\left(\displaystyle\sum_{k=1}^s\alpha_k\xi_k\right)=\displaystyle\sum_{k=1}^s\widetilde{\lambda}_k\alpha_k\xi_k$.
\end{center}
We also denote by $\Lambda_s$ the multiplier operator defined on $\R^s$ by 
\begin{equation*}
 \Lambda_s(\alpha_1,\ldots,\alpha_s)=\left(\widetilde{\lambda}_1\alpha_1,\ldots,\widetilde{\lambda}_s\alpha_s\right).
 \end{equation*}
For $\xi\in \mathcal{T}_{M_1,M_2} $ and $1\leq p\leq \infty$, we define
\begin{center}
 $\norm{\xi}{\Lambda_s,p}=\norm{\Lambda_s\xi}{p}$,
 \end{center}
 and for $\alpha\in\R^s$ we define
 \begin{center}
  $\norm{\alpha}{(\Lambda_s,p)}=\norm{J(\alpha)}{\Lambda_s,p}$.
 \end{center}
\noindent The application $\mathcal{T}_{M_1,M_2}  \ni\xi\longmapsto\norm{\xi}{\Lambda_s,p}$ is a norm on $\mathcal{T}_{M_1,M_2} $ and  the application $\R^s\ni\alpha\longmapsto\norm{\alpha}{(\Lambda_s,p)}$ is a norm on $\R^s$. We denote
\begin{center}
 $B_{\Lambda_s,p}^s=B_{\Lambda,p}^s=\{\xi \in \mathcal{T}_{M_1,M_2}  : \norm{\xi}{\Lambda_s,p}\leq1 \},$
\end{center}

\begin{center}
 $B_{(\Lambda_s,p)}^n=B_{(\Lambda,p)}^s=\{\alpha \in \C^s : \norm{\alpha}{(\Lambda_s,p)}\leq1 \}.$
\end{center}
If $\Lambda_s$ is the identity operator I, we will write $\norm{\cdot}{I,p}=\norm{\cdot}{p}$, $\norm{\cdot}{(I,p)}=\norm{\cdot}{(p)}$, $B_{I,p}^s=B_{p}^s$  and $B_{(I,p)}^s=B_{(p)}^s$.
\end{myrem}


\begin{myth}\label{theo3.1}
 Let $\lambda:[0,\infty)\rightarrow \R$ be a positive and monotonic function, $s=\dim\mathcal{T}_{M_1,M_2}$ and consider the orthonormal system $\{\xi_k\}^s_{k=1}$ 
 of $\mathcal{T}_{M_1,M_2}$ and the multiplier operator $\Lambda_s$ on $\mathcal{T}_{M_1,M_2}$ as in Remark \ref{myrem3.8}. If $\lambda$ is non-increasing, then there is an absolute 
 constant $C>0$ such that:
 \begin{itemize}
 \item[(a)] If $2\leq p<\infty$, we have
 \begin{equation}
 \nonumber s^{-1/2}\left(\displaystyle\sum_{l=M_1+1}^{M_2}(\lambda(l))^2 d_l\right)^{1/2}\leq M(\norm{\cdot}{(\Lambda_s,p)})\leq C\omega_d^{1/p-1/2}p^{1/2}s^{-1/2}\left(\displaystyle\sum_{l=M_1+1}^{M_2}(\lambda(l-1))^2 d_l\right)^{1/2}.
 \end{equation}
\item[(b)] If $p= \infty$, we have
\begin{equation}
\nonumber s^{-1/2}\left(\displaystyle\sum_{l=M_1+1}^{M_2}(\lambda(l))^2 d_l \right)^{1/2}\leq M(\norm{\cdot}{(\Lambda_s,\infty)})\leq C\omega_d^{-1/2}(\ln s)^{1/2}s^{-1/2}\left(\displaystyle\sum_{l=M_1+1}^{M_2}(\lambda(l-1))^2 d_l \right )^{1/2}.
 \end{equation}
 \item[(c)] If $1\leq p\leq 2$, we have
 \begin{equation*}
  \frac{\omega_d^{1/2}}{2}s^{-1/2}\left(\displaystyle\sum_{l=M_1+1}^{M_2}(\lambda(l))^2 d_l \right)^{1/2}\leq M(\norm{\cdot}{(\Lambda_s,p)})\leq s^{-1/2}\left(\displaystyle\sum_{l=M_1+1}^{M_2}(\lambda(l-1))^2 d_l \right )^{1/2}.
 \end{equation*}
  \item[(d)] If $ p= 2$, we have
 \begin{equation*}
  s^{-1/2}\left(\displaystyle\sum_{l=M_1+1}^{M_2}(\lambda(l))^2 d_l \right)^{1/2}\leq M(\norm{\cdot}{(\Lambda_s,p)})\leq s^{-1/2}\left(\displaystyle\sum_{l=M_1+1}^{M_2}(\lambda (l-1))^2d_l \right )^{1/2}.
 \end{equation*} 
 \end{itemize}
 If $\lambda$ is non-decreasing, then we obtain the estimates in $(a)$, $(b)$, $(c)$ and $(d)$ permuting $\lambda(l)$ for $\lambda(l-1)$.
 
\end{myth}
\begin{proof} Suppose $\lambda$ a non-increasing function. For a continuous function $f$ on $S^{n-1}$ consider the function $\widetilde{f}$ defined on $\R^n \backslash \{0\}$ by $\widetilde{f}(x)=|||x|||^2f\left(x/|||x|||\right)$. It is known that 
 \begin{equation}\label{(3.11)}
 \displaystyle\int_{S^{n-1}} f(x)  d\mu(x)= \frac{2\pi}{n} \displaystyle\int_{\R^n} \widetilde{f}(x)  d\gamma(x),
 \end{equation}
 where  $d\gamma(x)=e^{-\pi|||x|||^2}dx$ denotes the Gaussian measure on $\R^n$. Let $\{r_k\}_{k=1}^\infty$ be the sequence of Rademacher's functions given by $
 r_k(\theta)=\textrm{sign}\hspace{0.1cm} \sin(2^k\pi\theta)$, $\theta\in [0,1]$, $k\in\N$
 and let
 \begin{equation*}
 \delta_i^m(\theta)=m^{-1/2}(r_{(i-1)m+1}(\theta)+\cdots+r_{im}(\theta)),\hspace{0.5cm} m\in\N,\hspace{0.1cm}i=1,2,\ldots,n.
\end{equation*}
It follows by Lemma 2.1 in Kwapień \cite{Kwapien}, p. 585, that if $h:\R^n\longrightarrow \R $ is a continuous function satisfying\\ $h(x_1,\ldots,x_n)e^{-\sum_{k=1}^{n}|x_k|}\rightarrow0$,
uniformly when $\sum_{k=1}^n|x_k|\rightarrow \infty$, then
 \begin{equation}\label{(3.12)}
 \displaystyle\int_{\R^n} h(x)  d\gamma(x)=\displaystyle\lim_{m \to{}\infty}{\displaystyle\int_{0}^{1}h((2\pi)^{-1/2}(\delta_1^m(\theta),\ldots,\delta_n^m(\theta))) }d\theta.
\end{equation}
 Now, we consider $f(x)=\norm{x}{(\Lambda_s,p)}^2$, $x\in S^{s-1}$ and $h(x)=\widetilde{f}(x)=f(x)$, $x\in \R^{s}$ applying (\ref{(3.11)}) and (\ref{(3.12)}), we obtain
 \begin{eqnarray}\label{eq:(65)}
\nonumber\displaystyle\int_{S^{s-1}}\norm{x}{(\Lambda_s,p)}^2d\mu(x) & =&\frac{2\pi}{s}\displaystyle\lim_{m \to{}\infty}\displaystyle\int_{0}^{1}\norm{(2\pi)^{-1/2}(\delta_1^m(\theta),\ldots,\delta_{s}^m(\theta))}{(\Lambda_s,p)}^2d\theta \\
&=&s^{-1}\displaystyle\lim_{m \to{}\infty}\displaystyle\int_{0}^{1}\left(\displaystyle\int_{\Omega_d}\left|\displaystyle\sum_{i=1}^{s}\widetilde{\lambda}_i\delta_i^m(\theta)\xi_i(z)\right|^pd\sigma_d(z)\right)^{2/p}
d\theta.
\end{eqnarray} 
Denoting $\widetilde{\xi}_{(i-1)m+k}(z)=m^{-1/2}\xi_i(z)$, $z \in \Omega_d $ and $\widetilde{\widetilde{\lambda}}_{(i-1)m+k}=\widetilde{\lambda}_i$, for $i=1,\ldots,s;$ $k=1,\ldots,m$ and $m=1,2,\ldots$\\ we obtain 
\begin{equation}\label{eq:(67)}
\displaystyle\sum_{i=1}^{s}\widetilde{\lambda}_i\delta_i^m(\theta)\xi_i(z)=\displaystyle\sum_{j=1}^{ms}r_j(\theta)\widetilde{\widetilde{\lambda}}_j\widetilde{\xi}_j(z).
\end{equation}
Hence, by (\ref{eq:(65)}) and (\ref{eq:(67)}),
\begin{eqnarray}\label{eq:(3.28)}
\displaystyle\int_{S^{s-1}}\norm{x}{(\Lambda_s,p)}^2d\mu(x) &=&s^{-1}\displaystyle\lim_{m \to{}\infty} \displaystyle\int_{0}^{1}\left( \displaystyle\int_{\Omega_d}\left|\displaystyle\sum_{j=1}^{ms}r_j(\theta)\widetilde{\widetilde{\lambda}}_j\widetilde{\xi}_j(z)\right|^p d\sigma_d(z)\right)^{2/p}d\theta .
\end{eqnarray}
Using the addition formula we get
\begin{eqnarray*}
  \displaystyle\sum_{j=1}^{ms}|\widetilde{\widetilde{\lambda}}_j\widetilde{\xi}_j(z)|^2 &=&\displaystyle\sum_{i=1}^{s}\displaystyle\sum_{k=1}^{m}\left|\widetilde{\widetilde{\lambda}}_{(i-1)m+k}\widetilde{\xi}_{(i-1)m+k}(z)\right|^2=\displaystyle\sum_{i=1}^{s}\displaystyle\sum_{k=1}^{m}\left|\widetilde{\lambda}_i\right|^2\left|m^{-1/2}\xi_i(z)\right|^2\\
  &=&\displaystyle\sum_{l=M_1+1}^{M_2}\displaystyle\sum_{j=1}^{a_{l}}(\lambda_j^l)^2\displaystyle\sum_{i=1}^{d_{l,j}}\left|Y_i^{l,j}(z)\right|^2\\
  &\leq&\displaystyle\sum_{l=M_1+1}^{M_2}(\lambda(l-1))^2\displaystyle\sum_{j=1}^{a_{l}}\frac{d_{l,j}}{\omega_d}\\
  &=&\displaystyle\sum_{l=M_1+1}^{M_2}(\lambda(l-1))^2 \frac{d_l}{\omega_d}  
  \end{eqnarray*}
  and analogously,
  \begin{equation}
\nonumber \displaystyle\sum_{j=1}^{2ms}\left|\widetilde{\widetilde{\lambda}}_j\widetilde{\xi}_j(z)\right|^2 \geq \displaystyle\sum_{l=M_1+1}^{M_2}(\lambda(l))^2 \frac{d_l}{\omega_d}.
\end{equation}
Therefore, from Jensen's inequality, Khintchine's inequality (see \cite{Pietsch}, p. 41) and (\ref{eq:(3.28)}), we obtain for $2\leq p\leq\infty$,
\begin{eqnarray}\label{eq:(71)}
\nonumber M(\norm{\cdot}{(\Lambda_s,p)}))&\leq&s^{-1/2}\displaystyle\lim_{m \to{}\infty}\left( \left( \displaystyle\int_{\Omega_d}\displaystyle\int_{0}^{1}\left|\displaystyle\sum_{j=1}^{ms}r_j(\theta)\widetilde{\widetilde{\lambda}}_j\widetilde{\xi}_j(z)\right|^p d\theta d\sigma_d(z)\right)^{2/p} \right)^{1/2}\\ 
\nonumber&\leq& \gamma(p)s^{-1/2}\displaystyle\lim_{m \to{}\infty}\left(  \displaystyle\int_{\Omega_d}\left(\displaystyle\sum_{j=1}^{ms}\left|\widetilde{\widetilde{\lambda}}_j\widetilde{\xi}_j(z)\right|^2 \right)^{p/2} d\sigma_d(z) \right)^{1/p}.\\
\nonumber&\leq&\gamma(p)s^{-1/2}\displaystyle\lim_{m \to{}\infty}\left(  \displaystyle\int_{\Omega_d}\left(\displaystyle\sum_{l=M_1+1}^{M_2}(\lambda(l-1))^2\frac{d_l}{\omega_d}  \right)^{p/2} d\sigma_d(z) \right)^{1/p}\\
\nonumber&\leq&C_1p^{1/2}s^{-1/2}\omega_d^{1/p-1/2} \left(\displaystyle\sum_{l=M_1+1}^{M_2}(\lambda(l-1))^2d_l \right)^{1/2},
\end{eqnarray}
where $C_1$ is obtained from the fact that $\gamma(p)\asymp p^{1/2}$, and hence we get the upper estimate in $(a)$. On the other hand, for $p=1$, it follows from Khintchine's inequality, Jensen's inequality and (\ref{eq:(3.28)})
\begin{eqnarray*}
M(\norm{\cdot}{(\Lambda_s,1)})&=& s^{-1/2}\displaystyle\lim_{m \to{}\infty}\left( \displaystyle\int_{0}^{1}\left( \displaystyle\int_{\Omega_d}\left|\displaystyle\sum_{j=1}^{ms}r_j(\theta)\widetilde{\widetilde{\lambda}}_j\widetilde{\xi}_j(z)\right| d\sigma_d(z)\right)^2d\theta \right)^{1/2}\\
&\geq& s^{-1/2}\displaystyle\lim_{m \to{}\infty} \displaystyle\int_{\Omega_d}\displaystyle\int_{0}^{1} \left|\displaystyle\sum_{j=1}^{ms}r_j(\theta)\widetilde{\widetilde{\lambda}}_j\widetilde{\xi}_j(z)\right| d\theta d\sigma_d(z)\\ 
&\geq& \beta(1)s^{-1/2}\displaystyle\lim_{m \to{}\infty} \displaystyle\int_{\Omega_d} \left(\displaystyle\sum_{j=1}^{ms}\left|\widetilde{\widetilde{\lambda}}_j\widetilde{\xi}_j(z)\right|^2 \right)^{1/2}  d\sigma_d(z)\\ 
&\geq& \frac{1}{2}s^{-1/2}\omega_d^{1/2}\left(\displaystyle\sum_{l=M_1+1}^{M_2}(\lambda(l))^2d_l \right)^{1/2}.
\end{eqnarray*}
Since the Levy mean is an increasing function of p, it follows that 
 \begin{equation*}
 M(\norm{\cdot}{(\Lambda_s,p)})\geq M(\norm{\cdot}{(\Lambda_s,1)})\geq\frac{1}{2}s^{-1/2}\omega_d^{1/2} \left(\displaystyle\sum_{l=M_1}^{M_2}(\lambda(l))^2d_l \right)^{1/2},\hspace{0.4cm}1\leq p\leq2,
 \end{equation*}
 thus we obtain the lower estimate in $(c)$.\\
 \indent Now, we will obtain the inequalities in $(d)$. For $x\in\R^s$, we have
 \begin{eqnarray*}
 \norm{x}{(\Lambda_s,2)}^2&=&\displaystyle\sum_{l=M_1+1}^{M_2}\displaystyle\sum_{j=1}^{a_l}|\lambda_j^l|^2\displaystyle\sum_{i=1}^{d_{l,j}}(x_i^{l,j})^2\norm{Y_i^{l,j}}{2}^2\\
 &\leq&\displaystyle\sum_{l=M_1+1}^{M_2}\left(\displaystyle\sup_{1\leq j\leq a_l} \lambda_j^l\right)^2\displaystyle\sum_{j=1}^{a_l}\displaystyle\sum_{i=1}^{d_{l,j}}(x_i^{l,j})^2\\
 &\leq&\displaystyle\sum_{l=M_1+1}^{M_2}(\lambda(l-1))^2\displaystyle\sum_{j=1}^{a_l}\displaystyle\sum_{i=1}^{d_{l,j}}(x_i^{l,j})^2,
 \end{eqnarray*}
 and
 \begin{equation*}\label{eq:(62)}
  \displaystyle\int_{S^{s-1}}x_i^2d\mu(x)=\frac{1}{s}\displaystyle\sum_{i=1}^{s}\displaystyle\int_{S^{s-1}}x_i^2d\mu(z)=\frac{1}{s}\displaystyle\int_{S^{s-1}}|||x|||^2d\mu(x)=\frac{1}{s},\hspace{0.4cm}  i=1,\ldots,s.
 \end{equation*}
 Thus 
 \begin{eqnarray*}
\displaystyle\int_{S^{s-1}}\norm{x}{(\Lambda_s,2)}^2d\mu(z)&\leq&\displaystyle\sum_{l=M_1+1}^{M_2}(\lambda(l-1))^2\displaystyle\sum_{j=1}^{a_l}\displaystyle\sum_{i=1}^{d_{l,j}}\displaystyle\int_{S^{s-1}}(x_i^{l,j})^2d\mu(x)\\
&=&\frac{1}{s}\displaystyle\sum_{l=M_1+1}^{M_2}(\lambda(l-1))^2d_l, 
\end{eqnarray*}
and consequently
\begin{equation*}
  M(\norm{\cdot}{(\Lambda_s,2)})\leq s^{-1/2}\left(\displaystyle\sum_{l=M_1+1}^{M_2}(\lambda(l-1))^2d_l \right)^{1/2}.
\end{equation*}
Analogously we get
\begin{equation*}
  M(\norm{\cdot}{(\Lambda_s,2)})\geq s^{-1/2}\left(\displaystyle\sum_{l=M_1+1}^{M_2}(\lambda(l))^2d_l \right)^{1/2}.
\end{equation*}
Therefore we obtain $(d)$. Since the Levy mean $M(\norm{\cdot}{(\Lambda_s,p)})$ is an increasing function of $p$ for $1\leq p\leq\infty$, the lower estimates in $(a)$,$(b)$ and the 
upper estimate in $(c)$ follow from $(d)$.\\
\indent Finally, we will find the upper estimate in $(b)$. Any polynomial $t_s\in \mathcal{T}_{M_1,M_2}$, can be expressed as $t_s=D_{M_1,M_2}\ast t_s$, 
where $D_{M_1,M_2}=\sum_{(m,n)\in A_{M_2}\setminus A_{M_1}}Z_e^{(m,n)}$ and from Young's inequality, we get
\begin{equation*}
\norm{t_s}{\infty}=\norm{D_{M_1,M_2}\ast t_s}{\infty}\leq\norm{D_{M_1,M_2}}{\infty}\norm{t_s}{1},
\end{equation*}
and since $D_{M_1,M_2}=D_{M_1,M_2}\ast D_{M_1,M_2}$, then
\begin{equation*}
 \norm{D_{M_1,M_2}}{\infty}\leq\norm{D_{M_1,M_2}}{2}^2=\frac{s}{\omega_d}
\end{equation*}
and thus $\norm{t_s}{\infty}\leq s\norm{t_s}{1}/\omega_d$. Furthermore, if I denotes the identity operator, we have
\begin{equation*}
 \norm{I(t_s)}{\infty}\leq \frac{s}{\omega_d}\norm{t_s}{1},\hspace{0.5cm} \norm{I(t_s)}{\infty}\leq\norm{t_s}{\infty}.
\end{equation*}
Applying the Riesz-Thorin Interpolation Theorem to the above inequalities , we obtain
\begin{equation}\label{eq:(499)}
 \norm{t_s}{\infty}=\norm{I(t_s)}{\infty}\leq \left( \frac{s}{\omega_d}\right)^{1/p}\norm{t_s}{p},\hspace{2cm}1\leq p\leq\infty.
\end{equation}
Similarly, we can show that 
  \begin{equation}\label{eq:(17)}
   \norm{t_s}{p}\leq \left( \frac{s}{\omega_d}\right)^{1/2-1/p}\norm{t_s}{2}, \hspace{1cm}2\leq p\leq \infty. 
  \end{equation}
  To obtain the upper estimate in $(b)$, we apply (\ref{eq:(499)}) and the upper estimate in $(a)$ with $p=\ln s$, and we get 
\begin{eqnarray*}
 M(\norm{\cdot}{(\Lambda_s,\infty)})&\leq&\left(\frac{s}{\omega_d}\right)^{1/p}\left(\displaystyle\int_{S^{s-1}} \norm{x}{(\Lambda_s,p)}^2 d\mu(x)\right)^{1/2}\\
&\leq&\left(\frac{s}{\omega_d}\right)^{1/p}C_1p^{1/2}s^{-1/2}\omega_d^{1/p-1/2}\left(\displaystyle\sum_{l=M_1}^{M_2}(\lambda(l-1))^2d_l \right)^{1/2}\\
&=&eC_1(\ln s)^{1/2}s^{-1/2}\omega_d^{-1/2}\left(\displaystyle\sum_{l=M_1}^{M_2}(\lambda(l-1))^2d_l \right)^{1/2}.
\end{eqnarray*}
If $\lambda$ is a non-decreasing function, the proof is analogous.
  \end{proof}
\section{ Estimates for $n$-widths of general multiplier operators}
Consider a $n$-dimensional Banach space $E=(\R^n,\norm{\cdot}{})$. The dual norm of $\norm{\cdot}{}$ is defined by $\norm{x}{}^0=\sup\{|\langle x,y\rangle|:y\in B_E\}$, where $\langle x,y\rangle$ denotes the usual inner product of the elements $x,y\in \R^n$. The dual space 
$(\R^n,\norm{\cdot}{}^0)$ of $E$ will be denoted by $E^0$.
\begin{myth}(\cite{Pajor})\label{theo:(4.1)}
 There exists an absolute constant $C>0$ such that, for every $0<\rho<1$ and $n\in\N$, there exists a subspace $F_k\subseteq\R^n$, with 
 $\dim F_k=k>\rho n$ and 
 \begin{equation*}
  |||\alpha|||\leq CM\left(\norm{\cdot}{}^0\right)(1-\rho)^{-1/2}\norm{\alpha}{},\hspace{0.5cm}\alpha\in F_k.  
 \end{equation*}
\end{myth}


\begin{myth}\label{1}
 Let $1\leq q\leq p\leq2$, $0<\rho<1$, $s=\dim\mathcal{T}_{N}$, $\mathcal{T}_{N}=\bigoplus_{l=0 }^N\mathcal{H}_l$, $d_l=\dim \mathcal{H}_l$ 
 and let $\lambda:[0,\infty)\rightarrow \mathbb{R}$ be a positive and non-increasing function with $\lambda(t)\neq0$ for $t\geq0$
 and $\Lambda=\{\lambda_{m,n}\}_{m,n\in \mathbb{N}}$, $\lambda_{m,n}=\lambda(|(m,n)|)$. 
 Then there is an absolute constant $C>0$ such that 
 \begin{equation*}  
  \min\{d_{[\rho s-1]}(\Lambda U_p,L^q),d^{[\rho s-1]}(\Lambda U_p,L^q)\} \geq  C(1-\rho)^{1/2}s^{1/2}\left(\displaystyle\sum_{l=1}^{N}(\lambda(l))^{-2}d_l \right)^{-1/2}\vartheta_{q,s} 
   \end{equation*}
 where 
 \begin{center}
 $\vartheta_{q,s} = \left \{ \begin{matrix} (1-1/q)^{1/2}, & \mbox{ }q>1,
\\ (\ln s)^{-1/2}, & \mbox{ }q=1,\end{matrix}\right.$
\end{center}
and $[\rho s-1]$ denotes the integer part of the number $\rho s-1$.
 \end{myth}
 \begin{proof}
  Let $x,y\in\R^s=J^{-1}\mathcal{T}_{N}$. The Hölder's inequality implies that
  \begin{eqnarray*}
 \norm{x}{(\Lambda_s,q)}^0&=&\sup\{|\langle x,y\rangle|: y\in B_{(\Lambda_s,q)}^s\}\\
 &=&\sup\left\lbrace\left|\displaystyle\int_{\Omega_d} (Jx)(Jy) d\sigma_d\right|: Jy\in \Lambda_s^{-1}(B_q^s)\right\rbrace\\ 
 &=&\sup\left\lbrace\left|\displaystyle\int_{\Omega_d}(\Lambda_s^{-1}Jx)(J\overline{y}) d\sigma_d\right|: \overline{y}\in B_{(q)}^s\right\rbrace\\
 &\leq&\sup\{\norm{\Lambda_s^{-1}Jx}{q'}\norm{J\overline{y}}{q}: \overline{y}\in B_{(q)}^s\}\\
 &\leq&\norm{\Lambda_s^{-1}Jx}{q'} =\norm{x}{(\Lambda_s^{-1},q')}
 \end{eqnarray*}
 where $1/q+1/q'=1$ and $\Lambda^{-1}= \{\lambda_{m,n}^{-1}\}_{m,n\in \N}$. Taking $0<\rho<1$, the Theorem \ref{theo:(4.1)} guarantees the existence of 
 a subspace $F_k\subseteq \R^s$, $\dim F_k=k> \rho s$, such that 
 \begin{eqnarray*}
\norm{x}{(2)}=|||x|||&\leq& C'M\left(\norm{\cdot}{(\Lambda_s,q)}^0\right)(1-\rho)^{-1/2}\norm{x}{(\Lambda_s,q)},\hspace{1cm}\\
&\leq& C'M\left(\norm{\cdot}{(\Lambda_s^{-1},q')}\right)(1-\rho)^{-1/2}\norm{x}{(\Lambda_s,q)},\hspace{1cm}
\end{eqnarray*}
for all $x\in F_k$. Thus, for $\epsilon=(C')^{-1}(1-\rho)^{1/2}\left(M\left(\norm{\cdot}{(\Lambda_s^{-1},q')}\right)\right)^{-1}$, we have that $\epsilon B_{(\Lambda_s,q)}^s\cap F_k \subseteq B_{(2)}^s$.\\
From Theorem \ref{theo3.1}, since $2\leq q'\leq \infty$, it follows that 
\begin{equation}
\epsilon \geq C(1-\rho)^{1/2}\left \{ \begin{matrix} (q')^{-1/2}s^{1/2}\left(\displaystyle\sum_{l=1}^{N}(\lambda(l))^{-2}d_l \right)^{-1/2}, & \mbox{ }q'<\infty,
\\ (s/\ln s)^{1/2}\left(\displaystyle\sum_{l=1}^{N}(\lambda(l))^{-2}d_l \right)^{-1/2}, & \mbox{ }q'=\infty,\end{matrix}\right.\nonumber
\end{equation}
and since $1/q+1/q'=1$,
\begin{equation}
 \epsilon \geq C(1-\rho)^{1/2}\left \{ \begin{matrix} (1-1/q)^{1/2}s^{1/2}\left(\displaystyle\sum_{l=1}^{N}(\lambda(l))^{-2}d_l \right)^{-1/2}, & \mbox{ }q>1
\\ (s/\ln s)^{1/2}\left(\displaystyle\sum_{l=1}^{N}(\lambda(l))^{-2}d_k \right)^{-1/2}, & \mbox{ }q =1 .\end{matrix}\right.\nonumber
\end{equation}
We have that  $\Lambda_s=\Lambda\lvert_{\mathcal{T}_{N}}$ and since $B_2^s\subset U_2$, then $\Lambda_s(B_2^s)\subset\Lambda(U_2)$, therefore 
\begin{eqnarray*}
& &\min\{d_{[\rho s-1]}(\Lambda U_p,L^q),d^{[\rho s-1]}(\Lambda U_p,L^q)\}\\
 &\geq&\min\{d_{[\rho s-1]}(\Lambda U_2,L^q),d^{[\rho s-1]}(\Lambda U_2,L^q)\}\\
 &\geq& b_{[\rho s-1]}(\Lambda U_2,L^q)\\
 &\geq& b_{[\rho s-1]}(\Lambda_s B_2^s,L^q).
\end{eqnarray*}
Taking $X_k=J(F_k)$ we obtain $\epsilon B_q^s\cap X_k \subseteq \Lambda_s B_2^s$. Thus by the definition of Bernstein $n$-width
\begin{center}
 $b_{[\rho s-1]}(\Lambda_s B_2^s,L^q)\geq b_{k-1}(\Lambda_s B_2^s,L^q) \geq \epsilon$
\end{center}
and consequently
\begin{equation*}\label{eq:(3.2.3)}
 \min\{d_{[\rho s-1]}(\Lambda U_p,L^q),d^{[\rho s-1]}(\Lambda U_p,L^q)\}\geq \epsilon,
\end{equation*}
thus concluding the proof of the theorem.
 \end{proof}
 The next result follows applying Theorem \ref{1}, (\ref{(1.2)}) and basic properties of $n$-widths.
\begin{mycor}\label{cor1}
 In the conditions of Theorem \ref{1}, we have
 \begin{center}
  $d_{[\rho s-1]}(\Lambda U_p,L^p)\geq C'(1-\rho)^{1/2}s^{1/2}\left(\displaystyle\sum_{l=1}^{N}(\lambda(l))^{-2}d_l \right)^{-1/2}\kappa_s$,
 \end{center}
where 
\begin{center}
 $\kappa_s=\left \{ \begin{matrix} 1, & \mbox{ }1\leq p\leq 2, 1<q\leq 2,
 \\1,& \mbox{ }2\leq p< \infty, 2\leq q\leq \infty,
 \\1,& \mbox{ }1\leq p\leq2\leq q\leq \infty, 
\\ (\ln s)^{-1/2}, & \mbox{ }1\leq p\leq 2, q =1,
\\ (\ln s)^{-1/2}, & \mbox{ }p=\infty, 2\leq q\leq \infty. \end{matrix}\right.$
\end{center}
\end{mycor}
\begin{myth}\label{2}
Let $\lambda:(0,\infty)\longrightarrow \mathbb{R}$ be a positive and non-increasing function 
and let $\Lambda=\{\lambda_{m,n}\}_{m,n\in \mathbb{N}}$, $\lambda_{m,n}=\lambda(|(m,n)|)$. Suppose $1\leq p\leq 2\leq q\leq \infty $ and that the multiplier operator $\Lambda$ is bounded from $L^1$ to $L^2$. Let
 $\{N_k\}_{k=0}^\infty$ and $\{m_k\}_{k=0}^M$ be sequences of natural numbers such that $N_k<N_{k+1}$, $N_0=0$ and 
 $\sum_{k=0}^{M}m_k\leq \beta$. Then there exists an absolute constant $C>0$ such that
 \begin{center}
 $d_\beta(\Lambda U_p, L^q)\leq C \left(\displaystyle\sum_{k=1}^{M}\lambda(N_k)\varrho_{m_k}+\displaystyle\sum_{k=M+1}^{\infty}\lambda(N_k)\left(\theta_{N_k,N_{k+1}}\right)^{1/p-1/q}\right)$,
 \end{center}
 where 
\begin{equation*}
  \varrho_{m_k}=\frac{\theta_{N_k,N_{k+1}}^{1/p}}{(m_k)^{1/2}}\left \{ \begin{matrix} q^{1/2}, & \mbox{ }2\leq q<\infty,
\\ (\ln \theta_{N_k,N_{k+1}})^{1/2}, & \mbox{ }q=\infty,\end{matrix}\right.\hspace{0.4cm}
and \hspace{0.4cm}\theta_{N_k,N_{k+1}}=\displaystyle\sum_{j=N_k+1}^{N_{k+1}}\dim \mathcal{H}_{j}, k\geq 1.
\end{equation*}
\end{myth}
\begin{proof}
 Let $\mathcal{T}_{M_1,M_2}= \bigoplus_{l=M_1+1}^{M_2} \mathcal{H}_l$, $s=\dim\mathcal{T}_{M_1,M_2}=\sum_{l=M_1+1}^{M_2}d_l$, 
 $B_p^s=U_p\cap \mathcal{T}_{M_1,M_2} $, $B_{(p)}^s=J^{-1}B_p^s$, $\lambda(0)=0$ and $0<\rho<1$. The Theorem \ref{theo:(4.1)} guarantees 
 the existence of a subspace $F_k$ of $\R^s$, with $\dim F_k=k>\rho s$, such that, for all $\alpha\in F_k$,
 \begin{equation*}
  \norm{x}{(2)}\leq CM\left(\norm{\cdot}{(q)}\right)(1-\rho)^{-1/2}\norm{x}{(q)}^0.
 \end{equation*}
 If $m=s-k$, then $(1-\rho)^{-1/2}<\left(s/m\right)^{1/2}$. Applying the Theorem \ref{theo3.1} for $\Lambda_s=Id$, we get 
 \begin{equation*}\label{eq:(81)}
\norm{x}{(2)}\leq C\norm{x}{(q)}^0\left(\frac{s}{m}\right)^{1/2} \left \{ \begin{matrix}  q^{1/2}, & \mbox{ }2\leq q<\infty,
\\ (\ln s)^{1/2}, & \mbox{ }q =\infty. \end{matrix}\right.
\end{equation*}
Therefore, by (\ref{(1.2)})
\begin{eqnarray}\label{eq:(82)}
\nonumber d_m\left(B_2^s,L^q\cap \mathcal{T}_{M_1,M_2}\right)&=&d^m\left(\left(B_{(q)}^s\right)^0,(\R^s,\norm{\cdot}{(2)} )\right)\\
\nonumber &\leq&\displaystyle\sup_{x\in \left(B_{(q)}^s\right)^0\cap F_k} \norm{x}{(2)}\\
&\leq&C\left(\frac{s}{m}\right)^{1/2} \left \{ \begin{matrix}  q^{1/2}, & \mbox{ }2\leq q<\infty,
\\ (\ln s)^{1/2}, & \mbox{ }q =\infty. \end{matrix}\right.
\end{eqnarray}

Let $B_p^{N_k,N_{k+1}}=U_p\cap\mathcal{T}_{N_k,N_{k+1}}$ and for $f\in \Lambda U_p\subseteq L^2(\Omega_d)$, let $S_N(f)=\sum_{(m,n)\in A_N}f\ast \widetilde{Z}^{(m,n)} $, $\phi_{N_k,N_{k+1}}(f):=S_{N_{k+1}}(f)-S_{N_k}(f)$, $k\geq 1$, $\phi_{N_0,N_1}(f)=S_{N_1}(f)$. Since $S_{N_k}(f)\longrightarrow f$ in $L^2(\Omega_d)$, then
\begin{equation*}
 \displaystyle\lim_{j \to{}\infty}{\left\|\displaystyle\sum_{k=j}^\infty \phi_{N_k,N_{k+1}}(f)\right\|_2}=0.
\end{equation*}
Then it is easy to see that 
\begin{equation*}\label{eq:(83)}
 \Lambda U_p \subseteq \displaystyle\bigoplus_{k=0}^{\infty}(\phi_{N_k,N_{k+1}}\circ \Lambda) U_p,\hspace{0.4cm}1\leq p\leq 2.
\end{equation*}
But, for each $\varphi\in U_p$, $ \phi_{N_k,N_{k+1}}(\Lambda\varphi)=\Lambda(\phi_{N_k,N_{k+1}}(\varphi))$, therefore
\begin{equation}\label{eq:(84)}
 \Lambda U_p \subseteq \displaystyle\bigoplus_{k=0}^{\infty}(\Lambda\circ  \phi_{N_k,N_{k+1}}) U_p.
\end{equation}
Now, given $\varphi\in U_p$, 
\begin{eqnarray}\label{eq:(85)}
\nonumber \norm{(\Lambda\circ  \phi_{N_k,N_{k+1}})\varphi}{2} &\leq&\left\|\displaystyle\sum_{l=N_k+1}^{N_{k+1}}\lambda(l-1) \displaystyle\sum_{(m,n)\in A_l\setminus A_{l-1}} \varphi\ast \widetilde{Z}^{(m,n)}\right\|_{2}\\
 \nonumber &\leq&\lambda(N_k)\left\|\displaystyle\sum_{l=N_k+1}^{N_{k+1}}\displaystyle\sum_{(m,n)\in A_l\setminus A_{l-1}} \varphi\ast \widetilde{Z}^{(m,n)}\right\|_{2}\\
 &=&\lambda(N_k)\norm{\phi_{N_k,N_{k+1}}(\varphi)}{2},
 \end{eqnarray}
and by Young's inequality
\begin{eqnarray}\label{eq:(86)}
\norm{\phi_{N_k,N_{k+1}}(\varphi)}{2}&\leq&\norm{\varphi}{p}\left\|\displaystyle\sum_{(m,n)\in A_{N_{k+1}}\setminus A_{N_k}} Z_e^{(m,n)}\right\|_{1/(3/2-1/p)}. 
\end{eqnarray}
Using properties of the zonal harmonic functions, we have
\begin{eqnarray*}
 \left\|\displaystyle\sum_{(m,n)\in A_{N_{k+1}}\setminus A_{N_k}} Z_e^{(m,n)}\right\|_{2}^2&=&\displaystyle\sum_{(m,n)\in A_{N_{k+1}}\setminus A_{N_k}}\displaystyle\int_{\Omega_d}Z_e^{(m,n)}\overline{Z_e^{(m,n)}}d\sigma_d\\
 &=&\displaystyle\sum_{(m,n)\in A_{N_{k+1}}\setminus A_{N_k}} Z_e^{(m,n)}(e)\\
 &=&\displaystyle\sum_{(m,n)\in A_{N_{k+1}}\setminus A_{N_k}}\frac{d_{m,n}}{\omega_d}=\frac{\theta_{N_k,N_{k+1}}}{\omega_d}
\end{eqnarray*}
and thus by (\ref{eq:(86)}), for $p=1$ we obtain $ \norm{\phi_{N_k,N_{k+1}}(\varphi)}{2}\leq \omega_d^{-1/2}\theta_{N_k,N_{k+1}}^{1/2}\norm{\varphi}{1}$. Furthermore, for $p=2$, and $\varphi\in U_2$ we get $\norm{\phi_{N_k,N_{k+1}}(\varphi)}{2}\leq\norm{\varphi}{2}$. Applying the Riesz-Thorin Interpolation theorem to the last two inequalities, it follows for $1\leq p\leq 2$ that
\begin{equation*}
\norm{\phi_{N_k,N_{k+1}}\varphi}{2}\leq\omega_d^{1/2-1/p}\theta_{N_k,N_{k+1}}^{1/p-1/2}\norm{\varphi}{p}\leq\omega_d^{1/2-1/p}\theta_{N_k,N_{k+1}}^{1/p-1/2}.
\end{equation*}
Hence, from (\ref{eq:(85)}) and (\ref{eq:(84)}), 
\begin{equation}\label{eq:(87)}
 \Lambda U_p \subseteq \displaystyle\bigoplus_{k=0}^{\infty}\omega_d^{1/2-1/p}\lambda(N_k)\theta_{N_k,N_{k+1}}^{1/p-1/2}B_2^{N_k,N_{k+1}}.
\end{equation}
By (\ref{eq:(17)}) for $2\leq q\leq\infty$, we have for $\varphi\in B_2^{N_k,N_{k+1}}$ that
\begin{eqnarray*}
\norm{\varphi}{q}&=&\norm{\phi_{N_k,N_{k+1}}\varphi}{q}\leq\omega_d^{1/q-1/2}\theta_{N_k,N_{k+1}}^{1/2-1/q}\norm{\phi_{N_k,N_{k+1}}\varphi}{2}
 \leq\omega_d^{1/q-1/2}\theta_{N_k,N_{k+1}}^{1/2-1/q},
\end{eqnarray*}
and therefore $B_2^{N_k,N_{k+1}}\subseteq\omega_d^{1/q-1/2}\theta_{N_k,N_{k+1}}^{1/2-1/q}B_q^{N_k,N_{k+1}}$. Consequently by (\ref{eq:(87)}) 
\begin{equation}\label{eq (4.29)}
 \Lambda U_p\subseteq\displaystyle\bigoplus_{k=0}^{M}\omega_d^{1/2-1/p}\theta_{N_k,N_{k+1}}^{1/p-1/2}\lambda(N_k)B_2^{N_k,N_{k+1}}
+ \displaystyle\bigoplus_{k=M+1}^{\infty}\omega_d^{1/q-1/p}\lambda(N_k)\theta_{N_k,N_{k+1}}^{1/p-1/q}B_q^{N_k,N_{k+1}}.
\end{equation}

Finally, using (\ref{eq:(82)}), (\ref{eq (4.29)}) and properties of $n$-widths, we get 
\begin{eqnarray*}
 d_\beta(\Lambda U_p,L^q) &\leq&\displaystyle\sum_{k=0}^M\omega_d^{1/2-1/p}\lambda(N_k)\theta_{N_k,N_{k+1}}^{1/p-1/2}d_{m_k}(B_2^{N_k,N_{k+1}},L^q\cap \mathcal{T}_{N_k,N_{k+1}})\\
 &+& \displaystyle\sum_{k=M+1}^\infty\omega_d^{1/q-1/p}\lambda(N_k)\theta_{N_k,N_{k+1}}^{1/p-1/q}d_0(B_q^{N_k,N_{k+1}},L^q\cap \mathcal{T}_{N_k,N_{k+1}})\\
 &\leq&C\left(\displaystyle\sum_{k=0}^M\lambda(N_k)\varrho_{m_k}+\displaystyle\sum_{k=M+1}^\infty\lambda(N_k)\theta_{N_k,N_{k+1}}^{1/p-1/q}\right).
\end{eqnarray*}
\end{proof}
\begin{myrem}\label{rem4.1}
 We will improve the estimate obtained in the previous theorem, specifying the sequences $N_k$ and $m_k$. We define $N_1=N\in \N$ and
 \begin{equation*}
  N_{k+1}=\min\{l \in \N :e\lambda(l)\leq \lambda(N_k)\}.
 \end{equation*}
 Let $\theta_{N_k,N_{k+1}}$ be as in Theorem \ref{2}. For $\epsilon>0$ we define
 \begin{equation*}
  M=\left[\frac{\ln \theta_{N_1,N_2}}{\epsilon}\right],\hspace{0.3cm}m_k=[e^{-\epsilon k}\theta_{N_1,N_2}]+1,\hspace{0.3cm}k=1,\ldots,M,\hspace{0.3cm}m_0=\theta_{0,N}.
 \end{equation*}
 Hence,
 \begin{eqnarray}\label{eq:(89)}
 \displaystyle\sum_{k=1}^M m_k&=&\displaystyle\sum_{k=1}^M([e^{-\epsilon k}\theta_{N_1,N_2}]+1)
\leq M+\theta_{N_1,N_2}\displaystyle\sum_{k=0}^M(e^{-\epsilon })^k\leq C_\epsilon\theta_{N_1,N_2},
\end{eqnarray}
where $C_\epsilon>0$, depends only on $\epsilon$. Applying the previous theorem for
\begin{equation*}
 \beta=m_0+\displaystyle\sum_{k=1}^M m_k=\displaystyle\sum_{s=0}^N \dim \mathcal{H}_s+\displaystyle\sum_{k=1}^Mm_k ,
\end{equation*}
and writing $d_\beta=d_\beta(\Lambda U_p,L^q)$, we have 
\begin{eqnarray*}
 d_\beta&\leq&C\displaystyle\sum_{k=1}^M\lambda(N_k)\frac{\theta_{N_k,N_{k+1}}^{1/p}}{ m_k^{1/2}}\left \{ \begin{matrix}  q^{1/2}, & \mbox{ }2\leq q<\infty,
 \\ \left(\ln \theta_{N_k,N_{k+1}}\right)^{1/2}, & \mbox{ }q =\infty, \end{matrix}\right.\\
 &+&C\displaystyle\sum_{k=M+1}^\infty\lambda(N_k)\theta_{N_k,N_{k+1}}^{1/p-1/q}\\
 &\leq&eC\lambda(N)\displaystyle\sum_{k=1}^M e^{-k(1-\epsilon/2)}\frac{\theta_{N_k,N_{k+1}}^{1/p}}{ \theta_{N_1,N_2}^{1/2}}\left \{ \begin{matrix}  q^{1/2}, & \mbox{ }2\leq q<\infty,
 \\ \left(\ln \theta_{N_k,N_{k+1}}\right)^{1/2}, & \mbox{ }q =\infty, \end{matrix}\right.\\
 &+&eC\lambda(N)\displaystyle\sum_{k=M+1}^\infty e^{-k}\theta_{N_k,N_{k+1}}^{1/p-1/q}. 
\end{eqnarray*}
\end{myrem}

\begin{mydef}
 Let  $N_k,M$ and $\theta_{N_k,N_{k+1}}$ be as in the previous remark. We say that $\Lambda=\{\lambda_{m,n}\}_{m,n\in \N}\in K_{\epsilon,p}$, $\epsilon>0$, $1\leq p\leq 2$, if
 $\lambda(k+1)\leq\lambda(k)$, $N_{k+1}>N_k$, for all $k\in \N$ and if for all $N\in \N$, there exists a constant $C_{\epsilon,p}$, depending only on $d$, $\epsilon$ and $p$, such that
 \begin{equation}
  \displaystyle\sum_{k=1}^M e^{-k(1-\epsilon/2)}\frac{\theta_{N_k,N_{k+1}}^{1/p}}{ \theta_{N_1,N_2}^{1/2}}\leq C_{\epsilon,p}\theta_{N_1,N_2}^{1/p-1/2}.
 \end{equation}
\end{mydef}
\begin{mycor}\label{cor:(3.3.1)}
 Let $\Lambda=\{\lambda_{m,n}\}_{m,n\in \N}$ and $\theta_{N_k,N_{k+1}}$ be as in the previous theorem, and let $\epsilon> 0$, $N_k$, $M$, $\{m_k\}_{k=0}^M$ and $\beta$
  be as in the previos remark. Let $1\leq p\leq 2\leq q\leq \infty$ and suppose that $\Lambda\in K_{\epsilon,p}$, for a fixed $\epsilon>0$. 
  Then there exists a constant   $C_{\epsilon,p}>0$, such that
 \begin{eqnarray*}
  d_\beta(\Lambda U_p,L^q)&\leq&C_{\epsilon,p}\lambda(N)\theta_{N_1,N_2}^{1/p-1/2}\left \{ \begin{matrix}  q^{1/2}, & \mbox{ }2\leq q<\infty,
 \\ \displaystyle\sup_{1\leq k\leq M}\left(\ln \theta_{N_k,N_{k+1}}\right)^{1/2}, & \mbox{ }q =\infty, \end{matrix}\right.\\
 &+&C_{\epsilon,p}\lambda(N)\displaystyle\sum_{k=M+1}^\infty e^{-k}\theta_{N_k,N_{k+1}}^{1/p-1/q}.
 \end{eqnarray*}
\end{mycor}

\begin{myrem}
 We note that the Theorems \ref{1} and \ref{2}, and the Corollaries  \ref{cor1} and \ref{cor:(3.3.1)} hold if we consider  
 a function $\lambda$ such that $t\rightarrow |\lambda(t)|$ is a positive and non-increasing fuction. Simply we change  $\lambda(t)$ by $|\lambda(t)|$.
 \end{myrem}
 \section{Applications}
 Let $\lambda^{(1)},\lambda^{(2)}:[0,\infty)\rightarrow\R$ be defined by $\lambda^{(1)}(t)=t^{-\gamma}(\ln t)^{-\xi}$ for $t>1$, $\lambda^{(1)}(t)=0$ for $0\leq t\leq 1$, and $\lambda^{(2)}(t)=e^{-\gamma t^r}$, $t\geq 0$, where $\gamma,\xi,r\in\R$, $\gamma,r>0$, $\xi\geq 0$. In this section we consider the multiplier operators associated with the sequences $\Lambda^{(1)}=(\lambda_{m,n}^{(1)})_{m,n\in\N}$ and $\Lambda^{(2)}=(\lambda_{m,n}^{(2)})_{m,n\in\N}$ where $\lambda_{m,n}^{(1)}=\lambda^{(1)}(|(m,n)|)$, $\lambda_{m,n}^{(2)}=\lambda^{(2)}(|(m,n)|)$, $|(m,n)|=\max\{m,n\}$.
 \begin{myrem}\label{rem2.3}
  We will prove that, if $\gamma>(2d-1)/2$, then the multiplier operator $\Lambda^{(1)}$ is bounded from $L^1$ to $L^2$. Given $\varphi\in U_1$, for each $k\in \N^*$, let 
  \begin{equation*}
  \varphi_k=\displaystyle\sum_{l=1}^k\displaystyle\sum_{(m,n)\in A_l\setminus A_{l-1}}\varphi\ast \widetilde{Z}^{(m,n)}.
 \end{equation*}
 We obtain that \begin{equation*}
\Lambda^{(1)}\varphi_k=\displaystyle\sum_{l=1}^k\displaystyle\sum_{(m,n)\in A_l\setminus A_{l-1}}\lambda_{m,n}\varphi\ast \widetilde{Z}^{(m,n)},
\end{equation*}
and thus since $d_l\asymp l^{2d-2}$, we obtain 
\begin{eqnarray*}
\norm{\Lambda^{(1)}\varphi_k}{2}^2&=&\displaystyle\int_{\Omega_d}\Lambda^{(1)}\varphi_k\overline{\Lambda^{(1)}\varphi_k}d\sigma_d
=\displaystyle\sum_{l=1}^k\displaystyle\sum_{(m,n)\in A_l\setminus A_{l-1}}|\lambda_{m,n}|^2\norm{\varphi\ast \widetilde{Z}^{(m,n)}}{2}^2\\
&\leq&\displaystyle\sum_{l=1}^k|\lambda(l-1)|^2\displaystyle\sum_{(m,n)\in A_l\setminus A_{l-1}}\norm{\varphi\ast \widetilde{Z}^{(m,n)}}{2}^2
\leq\displaystyle\sum_{l=1}^k|\lambda(l-1)|^2\displaystyle\sum_{(m,n)\in A_l\setminus A_{l-1}}\frac{d_{m,n}}{\omega_d}\\
&=&\frac{1}{\omega_d}\displaystyle\sum_{l=1}^k\lambda^2(l-1)d_l
\leq C_1\displaystyle\sum_{l=1}^k(l-1)^{-2\gamma}(\ln(l-1))^{-2\xi}l^{2d-2}\\
 &\leq&C_2\displaystyle\sum_{l=1}^kl^{-2\gamma+2d-2}\leq C_3.
 \end{eqnarray*}
 Therefore 
 \begin{center}
 $\norm{\Lambda^{(1)}\varphi}{2}=\displaystyle\sup_k\norm{\Lambda^{(1)}\varphi_k}{2}\leq (C_3)^{1/2}$,
\end{center}
and thus $\Lambda^{(1)}$ is bounded from $L^1$ to $L^2$.
 \end{myrem}
 \begin{myth}\label{myth2.5}
For $1\leq p\leq\infty$, $2\leq q\leq \infty$, $\gamma>(2d-1)/p$ if $1\leq p\leq 2$ and $\gamma>(2d-1)/2$ 
 if $p\geq 2$,  and for all $m\in \N$, we have
 \begin{equation}
 d_m(\Lambda^{(1)}U_p,L^q)\ll m^{-\gamma/(2d-1)+(1/p-1/2)_+}(\ln m)^{-\xi}\left \{ \begin{matrix}  q^{1/2}, & \mbox{ }2\leq q<\infty,
 \\ (\ln m)^{1/2}, & \mbox{ }q =\infty. \end{matrix}\right.
\end{equation}
\end{myth}
\begin{proof}
 Suppose $1\leq p\leq 2\leq q\leq\infty$, fix $\delta>0$ and let $\lambda_1,\lambda_2:(0,\infty)\rightarrow\R$ be defined by $\lambda_1(t)=t^{-\gamma}$ and $\lambda_2(t)=t^{-\gamma-\delta}$. Let $a>1$ and let $b,b_1,b_2 \in \R$ 
 such that $e\lambda(b)=\lambda(a)$, $e\lambda_1(b_1)=\lambda_1(a)$ and $e\lambda_2(b_2)=\lambda_2(a)$. We have that $b,b_1,b_2>a$ and that $b_1=e^{1/\gamma}a$, $b_2=e^{1/(r+\delta)}a$ and 
 $b_1=b\left(\ln b/\ln a\right)^{\xi/\gamma}$. Since $b>a$, $b_1>b$ and $b>b_2$, we obtain $e^{1/(\gamma+\delta)}a<b<e^{1/\gamma}a$. 
 
 Taking $a=N_k$, we get $e^{1/(\gamma+\delta)}N_k<b<e^{1/\gamma}N_k$. But $N_k<b\leq N_{k+1}<b+1$ and therefore
\begin{equation}\label{eq:(2)}
 e^{1/(\gamma+\delta)}N_k\leq N_{k+1}\leq e^{1/\gamma}N_k+1,\hspace{0.5cm} k\geq1.
\end{equation}
Considering that $d_l=\dim \mathcal{H}_{l}\asymp l^{2d-2}$, integrating the function $x^{2d-2}$, we obtain
\begin{equation}\label{eq:(4)}
\theta_{N_k,N_{k+1}}\asymp \displaystyle\sum_{s=N_k+1}^{N_{k+1}}s^{2d-2}\leq \displaystyle\sum_{s=0}^{N_{k+1}-1}s^{2d-2}+N_{k+1}^{2d-2}\asymp N_{k+1}^{2d-1}.
\end{equation}
Furthermore using (\ref{eq:(2)}), we get
\begin{eqnarray}\label{eq:(3)}
\nonumber \theta_{N_k,N_{k+1}}&\asymp& \displaystyle\sum_{j=N_k+1}^{N_{k+1}}j^{2d-2}\\
\nonumber&\gg&N_{k+1}^{2d-1}-\frac{N_{k+1}^{2d-1}}{2^{2d-1}}\left(1+\frac{N_k}{N_{k+1}}\right)^{2d-1}\\
\nonumber&\geq& N_{k+1}^{2d-1}\left[1-2^{-2d+1}\left(1+(2d-1)e^{-1/\gamma+\delta}\right)\right]\\
\nonumber&+& N_{k+1}^{2d-1}\left[2^{-2d+1}\left(\frac{(2d-1)(2d-2)}{2!}e^{-2/\gamma+\delta}+\cdots+e^{(-2d+1)/(\gamma+\delta)}\right)\right]\\
&=& N_{k+1}^{2d-1}(1-C_{\gamma,\delta,d}).
\end{eqnarray}
Thus, by (\ref{eq:(3)}) and (\ref{eq:(4)}), it follows that
\begin{equation}\label{eq:(6)}
 \theta_{N_k,N_{k+1}}\asymp N_{k+1}^{2d-1},\hspace{0.5cm}k\geq 1,
\end{equation}
and using (\ref{eq:(2)}) we obtain
\begin{equation}\label{eq:(7)}
 e^{k/(\gamma+\delta)}N\leq N_{k+1}\leq C_\gamma e^{k/\gamma}N, \hspace{0.5cm}k\geq 1.
\end{equation}
Let $M=[\epsilon^{-1}\ln \theta_{N_1,N_2}]$ be as in Remark \ref{rem4.1}. We have that $ M\asymp \epsilon^{-1}\ln N\asymp\epsilon^{-1}\ln s$. 
By (\ref{eq:(6)}) and (\ref{eq:(7)}), 
\begin{eqnarray*}
 \sigma&=&\displaystyle\sum_{k=M+1}^{\infty}e^{-k}\left(\theta_{N_k,N_{k+1}}\right)^{1/p-1/q}\ll\displaystyle\sum_{k=M+1}^{\infty}e^{-k}\left(e^{k/\gamma}N\right)^{(2d-1)(1/p-1/q)}\\
 &=&N^{(2d-1)(1/p-1/q)}\displaystyle\sum_{k=M+1}^{\infty}e^{-k}e^{k(2d-1)(1/p-1/q)/\gamma}
 =N^{(2d-1)(1/p-1/q)}\displaystyle\sum_{k=M+1}^{\infty}e^{-k(1-(2d-1)(1/p-1/q)/\gamma)}.
\end{eqnarray*}
Since $M\asymp\epsilon^{-1}\ln N$, then there exists a positive constant $C$ such that 
\begin{equation}\label{eq:(8)}
 \sigma\ll N^{(2d-1)(1/p-1/q)}\displaystyle\sum_{k=[C\epsilon^{-1}\ln N]}^{\infty}e^{-k(1-(2d-1)(1/p-1/q)/\gamma)},
\end{equation}
and since $1\leq p\leq2\leq q \leq\infty$ and $\gamma>(2d-1)(1/p-1/q)$, we have that $1-(1/p-1/q)(2d-1)/\gamma>0$ and thus
\begin{eqnarray*}
\displaystyle\sum_{k=[C\epsilon^{-1}\ln N]}^{\infty}e^{-k(1-(2d-1)(1/p-1/q)/\gamma)}&\asymp&\frac{e^{-(1-(2d-1)(1/p-1/q)/\gamma)C\epsilon^{-1}\ln N}}{1-e^{-1(1-(2d-1)(1/p-1/q)/\gamma)}}\\
&\ll& N^{-C(1-(2d-1)(1/p-1/q)/\gamma)/\epsilon}.
\end{eqnarray*}
Therefore we obtain by (\ref{eq:(8)}) that 
\begin{equation*}
\sigma\ll N^{(2d-1)(1/p-1/q)-C(1-(2d-1)(1/p-1/q)/\gamma)/\epsilon}.
\end{equation*}
So, taking $0<\epsilon<C(1-(2d-1)(1/p-1/q)/\gamma)/(2d-1)(1/p-1/q)$, we get 
\begin{equation}\label{eq:(11)}
 \sigma \ll 1.
\end{equation}
Next we will prove that $\Lambda^{(1)}\in K_{\epsilon,p}$ for some $\epsilon>0$. Remembering that $M\asymp \epsilon^{-1}\ln N$, by  (\ref{eq:(6)}) and (\ref{eq:(7)}), it follows that 
\begin{eqnarray}\label{eq:(9)}
\nonumber\displaystyle\sum_{k=1}^{M}e^{-k(1-\epsilon/2)}\frac{\theta_{N_k,N_{k+1}}^{1/p}}{\theta_{N_1,N_2}^{1/2}}&\leq&\displaystyle\sum_{k=1}^{[C\epsilon^{-1}\ln N]}e^{-k(1-\epsilon/2)}\frac{\theta_{N_k,N_{k+1}}^{1/p}}{\theta_{N_1,N_2}^{1/2}}\\
\nonumber&\ll&\displaystyle\sum_{k=1}^{[C\epsilon^{-1}\ln N]}e^{-k(1-\epsilon/2)}\frac{(e^{k/\gamma}N)^{(2d-1)/p}}{(e^{1/(\gamma+\delta)}N)^{(2d-1)/2 }}\\
\nonumber&\ll& N^{(2d-1)(1/p-1/2)}\displaystyle\sum_{k=1}^{[C\epsilon^{-1}\ln N]}e^{-k(1-\epsilon/2-(2d-1)/\gamma p)}.
\end{eqnarray}
Since $\gamma/(2d-1)>1/p$, we have $t=-(1-\epsilon/2-(2d-1)/\gamma p)<0$ if $0<\epsilon/2<1-(2d-1)/\gamma p$ and consequently
\begin{equation*}\label{eq:(10)}
 \displaystyle\sum_{k=1}^{M}e^{-k(1-\epsilon/2)}\frac{\theta_{N_k,N_{k+1}}^{1/p}}{\theta_{N_1,N_2}^{1/2}}\leq C_{\epsilon,p}N^{(2d-1)(1/p-1/2)}.
\end{equation*}
By (\ref{eq:(7)})  there exists a constant $C_1$, such that $C_1e^{(2d-1)/(\gamma+\delta)}N^{2d-1}\leq\theta_{N_1,N_2}$ and then 
\begin{equation*}
 \displaystyle\sum_{k=1}^{M}e^{-k(1-\epsilon/2)}\frac{\theta_{N_k,N_{k+1}}^{1/p}}{\theta_{N_1,N_2}^{1/2}}\leq C_{\epsilon,p}^{\prime}\theta_{N_1,N_2}^{1/p-1/2},
\end{equation*}
that is, $\Lambda^{(1)}\in K_{\epsilon,p}$. Now, by Remark \ref{rem2.3} we have that $\Lambda^{(1)}$ is bounded from $L^1(\Omega_d)$ to $L^2(\Omega_d)$, 
and hence from Corollary \ref{cor:(3.3.1)} and (\ref{eq:(11)}), we obtain
\begin{equation*}
 d_{\beta}(\Lambda^{(1)}U_p,L^q)\ll \lambda(N)\theta_{N_1,N_2}^{1/p-1/2}\left \{ \begin{matrix} q^{1/2}, &\mbox{ } 2\leq q< \infty,\\ 
\displaystyle \sup_{1\leq k\leq M}(\ln \theta_{N_k,N_{k+1}})^{1/2}, & \mbox{ }q=\infty,\end{matrix}\right.
+\lambda(N).
\end{equation*}
For $1\leq k\leq M$, using (\ref{eq:(6)}), (\ref{eq:(7)}), the definition of $M$, and the fact that $s\asymp N^{2d-1}$, we obtain
\begin{eqnarray*}
 \theta_{N_k,N_{k+1}}&\ll& (e^{k/\gamma}N)^{2d-1}\leq (e^{M/\gamma}N)^{2d-1}\leq e^{(2d-1)(C \ln N)/\gamma\epsilon}N^{2d-1}\\
 &=& N^{(2d-1)+(2d-1)C /\gamma\epsilon}=\left(N^{2d-1}\right)^{1+C /\gamma\epsilon}\ll s^{1+C /\gamma\epsilon},
\end{eqnarray*}
and thus $ \ln (\theta_{N_k,N_{k+1}})\ll \ln (s^{1+C /\gamma\epsilon})=(1+C/\gamma\epsilon)\ln s\ll\ln s$. Furthermore, by  (\ref{eq:(6)}) and (\ref{eq:(7)}), it follows that $\theta_{N_1,N_2}^{1/p-1/2}\ll N^{(2d-1)(1/p-1/2)}$ and 
therefore
\begin{eqnarray*}
 d_\beta(\Lambda^{(1)}U_p,L^q)&\ll& \lambda(N)N^{(2d-1)(1/p-1/2)}\left \{ \begin{matrix} q^{1/2}, &\mbox{ } 2\leq q< \infty,\\ 
(\ln s)^{1/2}, & \mbox{ }q=\infty, \end{matrix}\right.\\
&=& N^{-\gamma}N^{(2d-1)(1/p-1/2)}(\ln N)^{-\xi}\left \{ \begin{matrix} q^{1/2}, &\mbox{ } 2\leq q< \infty,\\ 
(\ln s)^{1/2}, & \mbox{ }q=\infty. \end{matrix}\right.
\end{eqnarray*}
Now, since $s\asymp N^{2d-1}$, then $N^{-\gamma}\asymp s^{-\gamma/(2d-1)}$, $N^{(2d-1)(1/p-1/2)}\asymp s^{(1/p-1/2)}$ and 
$(\ln N)^{-\xi}\asymp(\ln N^d)^{-\xi}\asymp(\ln s)^{-\xi}$. Consequently
\begin{equation}\label{eq:(12)}
 d_\beta(\Lambda^{(1)}U_p,L^q)\ll s^{-\gamma/(2d-1)+(1/p-1/2)}(\ln s)^{-\xi}\left \{ \begin{matrix} q^{1/2}, &\mbox{ } 2\leq q< \infty,\\ 
(\ln s)^{1/2}, & \mbox{ }q=\infty. \end{matrix}\right.
\end{equation}
From Remark \ref{rem4.1}, $\beta\asymp s+\sum_{j=1}^{M}e^{-\epsilon j}\theta_{N_1,N_2}$ and keeping in mind that  $s\asymp N^{2d-1}$ and $M\asymp \epsilon^{-1}\ln N$, by (\ref{eq:(6)}) and (\ref{eq:(7)}) we get 
\begin{equation*}
 \beta\ll N^{2d-1}+e^{(2d-1)/\gamma}N^{2d-1} \displaystyle\sum_{j=1}^{[C\epsilon^{-1}\ln N]}(e^{-\epsilon})^j\ll N^{2d-1},
\end{equation*}
that is, there exists a constant $C_2\in\N$ such that $\beta\leq C_2N^{2d-1}$. Given $m\in \N$, let $N \in \N$ such that $C_2 N^{2d-1}\leq m\leq C_2 (N+1)^{2d-1}$. It follows by (\ref{eq:(12)}) that 
\begin{center}
 $d_m(\Lambda^{(1)}U_p,L^q)\ll m^{-\gamma/(2d-1)+(1/p-1/2)}(\ln m)^{-\xi}\left \{ \begin{matrix} q^{1/2}, &\mbox{ } 2\leq q< \infty,\\ 
(\ln m)^{1/2}, & \mbox{}q=\infty, \end{matrix}\right.$
\end{center}
for $1\leq p\leq 2$. The case $2\leq p\leq \infty$, follows because $\Lambda(U_p)\subset\Lambda(U_2)$.\end{proof}
\begin{myth}\label{myth2.6}
 For $\gamma>(2d-1)/2$ and for all $m\in \N$, we have
 \begin{equation}
  d_m(\Lambda^{(1)}U_p,L^q)\gg m^{-\gamma/(2d-1)}(\ln m)^{-\xi}\kappa_m,
 \end{equation}
 where 
\begin{center}
 $\kappa_m=\left \{ \begin{matrix} 1, & \mbox{ }1\leq p\leq 2, 1<q\leq 2,
 \\1,& \mbox{ }2\leq p< \infty, 2\leq q\leq \infty,
 \\1,& \mbox{ }1\leq p\leq2\leq q\leq \infty, 
\\ (\ln m)^{-1/2}, & \mbox{ }1\leq p\leq 2, q =1,
\\ (\ln m)^{-1/2}, & \mbox{ }p=\infty, 2\leq q\leq \infty. \end{matrix}\right.$
\end{center}
\end{myth}
\begin{proof}
We have that $d_l\asymp l^{2d-2}$ and $s=\dim\mathcal{T}_N\asymp N^{2d-1}$, and thus 
\begin{eqnarray*}
  \left(\displaystyle\sum_{l=1}^N\left(\lambda(l)\right)^{-2}d_l\right)^{-1/2}&\asymp&\left(\displaystyle\sum_{l=1}^N(l^{-\gamma}(\ln l)^{-\xi})^{-2}l^{2d-2}\right)^{-1/2}
  \geq\left(\displaystyle\sum_{l=1}^NN^{2\gamma+2d-2}(\ln N)^{2\xi}\right)^{-1/2}\\
  &\asymp&s^{-\gamma/(2d-1)-1/2}(\ln s)^{-\xi}.
 \end{eqnarray*}
 Applying the Corollary \ref{cor:(3.3.1)} with $\rho=1/2$,  we obtain
 \begin{eqnarray*}
 d_{[(s-2)/2]}(\Lambda^{(1)}U_p,L^q)&\geq& C (1-1/2)^{1/2}s^{1/2}\left(\displaystyle\sum_{l=1}^{N}\left(\lambda(l)\right)^{-2}d_l \right)^{-1/2}\kappa_s\\
&\gg&s^{1/2}s^{-\gamma/(2d-1)-1/2}(\ln s)^{-\xi}\kappa_s
=s^{-\gamma/(2d-1)}(\ln s)^{-\xi}\kappa_s.
\end{eqnarray*}
Finally, taking $m=[s/3]\leq[(s-2)/2]$ we have that $m\asymp s$ and thus
\begin{equation}
\nonumber d_m(\Lambda^{(1)}U_p,L^q)\geq d_{[(s-2)/2]}(\Lambda^{(1)}U_p,L^q)\geq s^{-\gamma/(2d-1)}(\ln s)^{-\xi}\kappa_s\gg m^{-\gamma/(2d-1)}(\ln m)^{-\xi}\kappa_m.
\end{equation}\end{proof}
\begin{myrem}\label{rem2.4}
  We will prove that the multiplier operator $\Lambda^{(2)}$ is bounded from $L^p$ to $L^q$, for $1\leq p,q\leq\infty$.
  We will show that this result is true for $p=1$ and $q=\infty$ and the other cases will follow immediately from inequalities between norms. Given $\varphi\in U_1$, we have that 
   \begin{equation*}
  \Lambda^{(2)}\varphi\sim\displaystyle\sum_{l=1}^{\infty}\displaystyle\sum_{(m,n)\in A_l\setminus A_{l-1}}\lambda_{m,n}\varphi\ast \widetilde{Z}^{(m,n)}
 \end{equation*}
 and thus 
 \begin{eqnarray*}
 \norm{\Lambda^{(2)}\varphi}{\infty}&\leq&\displaystyle\sup_{\omega\in \Omega_d}\displaystyle\sum_{l=1}^{\infty}\left|\displaystyle\sum_{(m,n)\in A_l\setminus A_{l-1}}\lambda_{m,n}\varphi\ast \widetilde{Z}^{(m,n)}(\omega)\right|\leq\displaystyle\sum_{l=1}^{\infty}\norm{a_l}{\infty},
 \end{eqnarray*}
where $a_l(\omega)=\sum_{(m,n)\in A_l\setminus A_{l-1}}\lambda_{m,n}\varphi\ast \widetilde{Z}^{(m,n)}(\omega)$. Let 
\begin{equation*}
 D_l(\omega)=\displaystyle\sum_{(m,n)\in A_l\setminus A_{l-1}}\lambda_{m,n} Z_e^{(m,n)}(\omega),\hspace{0.5cm}\overline{D_l}(\omega)=\displaystyle\sum_{(m,n)\in A_l\setminus A_{l-1}}\lambda^{1/2}_{m,n} Z_e^{(m,n)}(\omega).
\end{equation*}
We have that $D_l=\overline{D_l}\ast\overline{D_l}$ and $a_l=\varphi\ast D_l$. It follows by Young's inequality that 
\begin{eqnarray*}
 \norm{D_l}{\infty}&=&\norm{\overline{D_l}\ast\overline{D_l}}{\infty}\leq\norm{\overline{D_l}}{2}\norm{\overline{D_l}}{2}\\
 &=&\displaystyle\sum_{(m,n)\in A_l\setminus A_{l-1}}\lambda_{m,n}\frac{d_{m,n}}{\omega_d}\\
 &\leq&\frac{\lambda(l-1)}{\omega_d}d_l\leq C_1\lambda(l-1)l^{2d-2}
\end{eqnarray*}
and hence 
\begin{equation}
 \norm{a_l}{\infty}=\norm{\varphi\ast D_l}{\infty}\leq\norm{\varphi}{1}\norm{D_l}{\infty}\leq C_1\lambda(l-1)l^{2d-2}.\nonumber
\end{equation}
 Moreover there is a constant $C_2$ such that $e^{-\gamma l^r}l^{2d}\leq C_2$, for all $l\in \N$ and therefore 
 \begin{equation}
 \norm{a_l}{\infty}\leq C_1\lambda(l-1)l^{2d-2}=C_1 e^{-\gamma (l-1)^r}l^{2d}l^{-2}\leq C_3 e^{-\gamma (l-1)^r}(l-1)^{2d}l^{-2}\leq C_4 l^{-2}.\nonumber
\end{equation}
Thus
\begin{equation}
 \norm{\Lambda^{(2)}\varphi}{\infty}\leq\displaystyle\sum_{l=1}^{\infty}\norm{a_l}{\infty}\leq\displaystyle\sum_{l=1}^{\infty}C_4 l^{-2}=C_5,\nonumber
\end{equation}
so $\Lambda^{(2)}$ is bounded from $L^1(\Omega_d)$ to $L^\infty(\Omega_d)$ .
 \end{myrem}
\begin{myth}\label{myth2.7}
 Consider the sequences $\phi_k=\dim \mathcal{T}_k$, $\psi_k=\phi_k-\phi_k^{1-r/(2d-1)}-1$ 
 and $\kappa_k$ as in Theorem \ref{myth2.6}. Then 
 \begin{equation}\label{eq:(4.2.1)}
 d_{[\psi_k]}(\Lambda^{(2)}U_p,L^q)\gg e^{-\mathcal{R}\phi_k^{r/(2d-1)}}\kappa_k,\hspace{0.3cm}r>0,\hspace{0.2cm} k\in\N,
\end{equation}
\begin{equation}\label{eq:(4.2.2)}
 d_{[\psi_k]}(\Lambda^{(2)}U_p,L^q)\gg e^{-\mathcal{R}\psi_k^{r/(2d-1)}}\kappa_k,\hspace{0.3cm}0<r\leq 2d-1,\hspace{0.2cm} k\in\N,
\end{equation}
\begin{equation}\label{eq:(4.2.3)}
 d_{\phi_k-1}(\Lambda^{(2)}U_p,L^q)\gg e^{-\mathcal{R}\phi_k^{r/(2d-1)}}\kappa_k,\hspace{0.3cm}r> 2d-1,\hspace{0.2cm} k\in\N,
\end{equation}
\begin{equation}\label{eq:(4.2.4)}
 d_k(\Lambda^{(2)}U_p,L^q)\gg e^{-\mathcal{R}k^{r/(2d-1)}}\kappa_k,\hspace{0.3cm}0<r\leq 1,\hspace{0.2cm} k\in\N,
\end{equation}
where
\begin{equation*}
 \mathcal{R}=\gamma\left(\frac{d!(d-1)!}{2}\right)^{r/(2d-1)}.
\end{equation*}
\end{myth}
\begin{proof}
For  $\theta,\eta,r>0$ and $\eta\geq r-1$, we have
 \begin{eqnarray}\label{eq:(4.2.5)}
  \displaystyle\sum_{k=1}^{N}e^{\theta k^r}k^\eta=\displaystyle\sum_{k=1}^{N}\left( \frac{1}{\theta r}k^{1-r+\eta} \right) e^{\theta k^r}\theta r k^{r-1}
  \ll N^{1-r+\eta}\displaystyle\int_1^{N+1}e^{\theta x^r}\theta r x^{r-1}dx
  \leq N^{1-r+\eta}e^{\theta (N+1)^r}.
 \end{eqnarray}
Since there is a positive constant $C_1$ such that $(N+1)^r\leq N^r+C_1N^{-1}$, then
\begin{equation*}
 e^{\theta(N+1)^r}\leq  e^{\theta C_r}e^{\theta N^r}=C_2e^{\theta N^r}.
\end{equation*}
Let us fix $N\in \N$ and let $s=\dim\mathcal{T}_N$, $D=2/d!(d-1)!$. By the Proposition \ref{pr:(3.1.1)} we have 
\begin{equation*}\label{eq:(4.2.7)}
  DN^{2d-1}\leq s\leq D N^{2d-1}+C_3N^{2d-2}.
 \end{equation*}
 From the above estimates, we obtain $N^{r}\leq D^{-r/(2d-1)}s^{r/(2d-1)}$ and $N^r\geq D^{-r/(2d-1)} s^{r/(2d-1)}-C_4N^{r-1} $ for a positive constant $C_4$ and for $r>0$. Therefore, for $\mathcal{R}=\gamma D^{-r/(2d-1)}$ and the fact that $s\asymp N^{2d-1}$, we get
 \begin{equation}\label{eq:(4.2.12)}
 -\mathcal{R} s^{r/(2d-1)}\leq -\gamma N^{r}\leq -\mathcal{R} s^{r/(2d-1)}+C_5s^{(r-1)/(2d-1)},
\end{equation}
for $r>0$. The above estimates also hold if we put $N+1$ in the place of $N$. If $\rho=1-s^{-r/(2d-1)}$, using (\ref{eq:(4.2.5)}), we get
\begin{eqnarray*}
 (1-\rho)^{1/2}s^{1/2}\left(\displaystyle\sum_{l=1}^N(\lambda(l))^{-2}d_l\right)^{-1/2}&\gg&s^{-r/2(2d-1)+1/2}\left(N^{1-r+2d-2}e^{2\gamma N^r}\right)^{-1/2}
 \gg e^{-\gamma N^r},
\end{eqnarray*}
 and hence it follows by Corollary \ref{cor:(3.3.1)} and (\ref{eq:(4.2.12)}) that
 \begin{eqnarray*}
 d_{[s-s^{1-r/(2d-1)}-1]}(\Lambda^{(2)}U_p,L^q)&=&d_{[\rho s-1]}(\Lambda^{(2)}U_p,L^q)
 \gg e^{-\gamma N^r}\kappa_s
 \gg e^{-\mathcal{R} s^{r/(2d-1)}}\kappa_s.
\end{eqnarray*}
The above estimate says that, for $r>0$
\begin{equation*}
 d_{[\psi_N]}(\Lambda^{(2)}U_p,L^q)\gg e^{-\mathcal{R} \phi_N^{r/(2d-1)}}\kappa_{\phi_N},
\end{equation*}
but, since $\phi_N=\dim\mathcal{T}_N \asymp N^{2d-1}$, we have that $\kappa_{\phi_N}\asymp\kappa_{N}$ and therefore we get (\ref{eq:(4.2.1)}). 
Furthermore, if $r> 2d-1$, then $[\psi_N]=[\phi_N-\phi_N^{1-{r/(2d-1)}}-1]=\phi_N-1$ and  thus (\ref{eq:(4.2.3)}) follows from (\ref{eq:(4.2.1)}).\\
\indent Consider $0<r\leq2d-1$. Then
\begin{eqnarray*}\label{eq:(4.2.11)}
\left(1-s^{-r/(2d-1)}\right)^{r/(2d-1)}=1-s^{-r/(2d-1)}\frac{r}{2d-1}\left(1+\frac{2d-1-r}{2!(2d-1)}s^{-r/(2d-1)}+\cdots\right)
 =1-s^{-r/(2d-1)}S_s,
\end{eqnarray*}
and there exists a positive constant $C_r$ such that $0<S_s\leq C_r$ for all $s\in\N$. Therefore
\begin{eqnarray*}
 (\psi_N+1)^{r/(2d-1)}&=&\left(s-s^{1-r/(2d-1)}\right)^{r/(2d-1)}=s^{r/(2d-1)}\left(1-s^{-r/(2d-1)}\right)^{r/(2d-1)}\\
 &\geq&s^{r/(2d-1)}\left(1-C_r s^{-r/(2d-1)}\right)=s^{r/(2d-1)}-C_r=\phi_N^{r/(2d-1)}-C_r,
\end{eqnarray*}
thus 
\begin{equation}\label{eq:(4.2.13)}
 e^{-\mathcal{R}(\psi_N+1)^{r/(2d-1)}}\leq e^{-\mathcal{R}(\phi_N^{r/(2d-1)}-C_r)}\ll e^{-\mathcal{R}\phi_N^{r/(2d-1)}}.
\end{equation}
and from (\ref{eq:(4.2.1)}) we get 
\begin{equation}\label{eq:(4.2.15)}
 d_{[\psi_N]}(\Lambda^{(2)}U_p,L^q)\gg e^{-\mathcal{R}(\psi_N+1)^{r/(2d-1)}}\kappa_{N}.
\end{equation}
Now let $0<r\leq 1$. Using the Mean Value Theorem we get $(N+1)^r-N^r=r(N+c)^{r-1}<rN^{r-1}\leq r$, and then 
$1<e^{-\gamma N^r}/e^{-\gamma(N+1)^r}\leq e^{\gamma r}$.
From (\ref{eq:(4.2.1)}) and (\ref{eq:(4.2.12)}) it follows that 
\begin{eqnarray*}
 d_{[\psi_{N+1}]}(\Lambda^{(2)}U_p,L^q)&\gg&e^{-\gamma(N+1)^r}\kappa_{N+1}\gg e^{-\gamma N^r}\kappa_{N}\\
 &\gg& e^{-\mathcal{R}s^{r/(2d-1)}}\kappa_{N}\\
 &=&e^{-\mathcal{R}\phi_N^{r/(2d-1)}}\kappa_{N},
\end{eqnarray*}
and using (\ref{eq:(4.2.13)}) we obtain 
\begin{equation}\label{eq:(4.2.17)}
 d_{[\psi_{N+1}]}(\Lambda^{(2)}U_p,L^q)\gg e^{-\mathcal{R}\phi_N^{r/(2d-1)}}\kappa_{N}
\gg e^{-\mathcal{R}(\psi_N+1)^{r/(2d-1)}}\kappa_{N}.
\end{equation}
If $k\in\N$ and $[\psi_N]<k\leq[\psi_{N+1}]$, then  from (\ref{eq:(4.2.17)})  
\begin{eqnarray*}
 d_k(\Lambda^{(2)}U_p,L^q)&\gg&d_{[\psi_{N+1}]}(\Lambda^{(2)}U_p,L^q)\gg e^{-\mathcal{R}(\psi_N+1)^{r/(2d-1)}}\kappa_{N}\\
 &\gg&e^{-\mathcal{R}\psi_N^{r/(2d-1)}}\kappa_{N}\geq e^{-\mathcal{R}k^{r/(2d-1)}}\kappa_{N}\\
 &\gg&e^{-\mathcal{R}k^{r/(2d-1)}}\kappa_{k}
\end{eqnarray*}
and thus we proved (\ref{eq:(4.2.4)}).\\
\indent Finally, if $0<r\leq 2d-1$, we get $e^{\mathcal{R}l^{r/(2d-1)}}\asymp e^{\mathcal{R}(l+1)^{r/(2d-1)}}$, $l\in\N$ and 
thus it follows from (\ref{eq:(4.2.15)}) that
\begin{equation*}
 d_{[\psi_N]}(\Lambda^{(2)}U_p,L^q)\gg e^{-\mathcal{R}(\psi_N+1)^{r/(2d-1)}}\kappa_{N}\gg e^{-\mathcal{R}\psi_N^{r/(2d-1)}}\kappa_{N},
\end{equation*}
that is, we proved (\ref{eq:(4.2.2)}).
\end{proof}
\begin{myth}\label{myth2.8}
 Let $\phi_k$ be as in Theorem \ref{myth2.7}. Then for $0< r \leq 1$, $1\leq p\leq\infty$, $2\leq q\leq\infty$ 
 and for all $k\in \N$, we have
 \begin{equation}\label{eq:(2.0.0000)}
 d_k(\Lambda^{(2)}U_p,L^q) \ll e^{-\mathcal{R}k^{r/(2d-1)}}k^{(1-r/(2d-1))(1/p-1/2)_+}\left \{ \begin{matrix} 1, & \mbox{ }2\leq q<\infty,
\\ (\ln k)^{1/2}, & \mbox{ }q=\infty, \end{matrix}\right.
 \end{equation}
and for all  $r>1$ and all $k\in\N$,
\begin{equation}\label{eq:(2.0.00)}
 d_{\phi_k}(\Lambda^{(2)}U_p,L^q) \ll e^{-\gamma k^r}\left \{ \begin{matrix} k^{(2d-2)(1/p-1/q)}, & \mbox{ }1\leq p\leq2\leq q\leq\infty,
\\ k^{(2d-2)(1/2-1/q)}, & \mbox{ }2\leq p,q\leq\infty, \end{matrix}\right.
 \end{equation}
where $\mathcal{R}$ is the constant given in Theorem \ref{myth2.7}.
\end{myth}
\begin{proof}
Suppose $0<r\leq 1$ and let $\lambda(x)=e^{-\gamma x^r}$. For $k\in\N$, let $x_k\in\R$, $e\lambda(x_k)=\lambda(N_k)$. Then $x_k=(N_k^r+1/\gamma)^{1/r}$ and $N_{k}<x_{k}\leq N_{k+1}<x_{k}+1$. Therefore
\begin{equation}\label{eq:(4.2.48)}
  (N_k^r+1/\gamma)^{1/r}\leq N_{k+1}\leq (N_k^r+1/\gamma)^{1/r}+1.
 \end{equation}
 We have that
 \begin{eqnarray*}
 N_{k+1}^r&\leq&(N_k^r+1/\gamma)\left[1+r(N_k^r+1/\gamma)^{-1/r}+\frac{r(r-1)}{2!}(N_k^r+1/\gamma)^{-2/r}+\cdots\right]\\
 &\leq&(N_k^r+1/\gamma)\left[1+C_1(N_k^r+1/\gamma)^{-1/r} \right]
 =(N_k^r+1/\gamma)+C_1(N_k^r+1/\gamma)^{1-1/r}
 \end{eqnarray*}
 and since $0<r\leq1$, then $(N_k^r+1/\gamma)^{1-1/r}\leq 1$ and therefore $ N_{k+1}^r\leq N_k^r+C_1$. Thus 
\begin{equation*}\label{eq:(4.2.49)}
 N_{k}^r+\frac{1}{\gamma}\leq N_{k+1}^r\leq N_k^r+C_1
\end{equation*}
and using that $N_1=N$, we show that
\begin{equation}\label{eq:(4.2.50)}
 N^r+\frac{k}{\gamma}\leq N_{k+1}^r\leq N^r+C_1k.
\end{equation}
By (\ref{eq:(4.2.48)}), $N_{k+1}\asymp (N_k^r+1/\gamma)^{1/r}=N_k\left(1+N_k^{-r}/\gamma\right)^{1/r}$, and  then
\begin{equation}
\nonumber N_{k+1}\asymp N_k\left(1+\frac{1}{r\gamma}N_k^{-r}+\frac{1-r}{2r^2\gamma^2}N_k^{-2r}+\frac{(1-r)(1-2r)}{6r^3\gamma^3}N_k^{-3r}+\cdots \right)^{1/r}.
\end{equation}
Thus $ N_{k+1}-N_{k}\asymp N_k^{1-r}$. On the other hand
\begin{equation}\label{eq:(4.2.51)}
 \theta_{N_k,N_{k+1}}\asymp N_{k}^{2d-r-1},
\end{equation}
consequently, from (\ref{eq:(4.2.50)}), (\ref{eq:(4.2.51)}) and for $\epsilon>0$ and $k$ sufficiently large,  we obtain 
\begin{eqnarray*}
 \left(\frac{\theta_{N_k,N_{k+1}}}{\theta_{N_1,N_{2}}}\right)^{r/(2d-r-1)}&\asymp&\left(\frac{N_k^{2d-r-1}}{N^{2d-r-1}}\right)^{r/(2d-r-1)}=N^{-r}N_k^{r}\\
 &\leq&1+C_1kN^{-r}
  \leq 1+C_1k.
\end{eqnarray*}
Hence, for any $\delta>0$, there exists $C_2>0$ such that
\begin{equation*}
 \left(\frac{\theta_{N_k,N_{k+1}}}{\theta_{N_1,N_{2}}}\right)^{r/(2d-r-1)}\leq C_2e^{\delta k}, \hspace{0.3cm}k\geq 0,
\end{equation*}
and taking $\delta'=\delta(2d-r-1)/r$, we get
\begin{equation}\label{eq:(4.2.52)}
 \frac{\theta_{N_k,N_{k+1}}}{\theta_{N_1,N_{2}}}\leq C_2'e^{\delta' k},\hspace{0.3cm}k\geq 0.
\end{equation}
For $1\leq p\leq 2$,
\begin{eqnarray*}
 \displaystyle\sum_{k=1}^{M}e^{-k(1-\epsilon)} \left(\frac{\theta_{N_k,N_{k+1}}}{\theta_{N_1,N_{2}}}\right)^{1/p}\leq(C_2')^{1/p}\displaystyle\sum_{k=1}^{M}e^{-k(1-\epsilon)}e^{\delta' k/p}
 \leq(C_2')^{1/p}\displaystyle\sum_{k=1}^{\infty}e^{-k(1-\epsilon-\delta'/p)}\leq C_3,
\end{eqnarray*}
 where $\epsilon$, $p$ e $\delta$ are chosen satisfying $\epsilon+\delta'/p<1$. Thus 
 \begin{equation*}\label{eq:(4.2.53)}
 \displaystyle\sum_{k=1}^{M}e^{-k(1-\epsilon)} \frac{\theta_{N_k,N_{k+1}}^{1/p}}{\theta_{N_1,N_{2}}^{1/2}}=\displaystyle\sum_{k=1}^{M}e^{-k(1-\epsilon)} \left(\frac{\theta_{N_k,N_{k+1}}}{\theta_{N_1,N_{2}}}\right)^{1/p}\theta_{N_1,N_{2}}^{1/p-1/2}\leq C_3\theta_{N_1,N_{2}}^{1/p-1/2},
\end{equation*}
and then $\Lambda^{(2)}\in K_{2\epsilon,p}$. Therefore it follows from Corollary \ref{cor:(3.3.1)} that 
\begin{eqnarray}\label{eq:(4.2.54)}
\nonumber d_\beta(\Lambda^{(2)}U_p,L^q) &\leq& C_{\epsilon,p}\lambda(N)\theta_{N_1,N_{2}}^{1/p-1/2}\left \{ \begin{matrix} q^{1/2}, & \mbox{ }2\leq q<\infty,
\\\displaystyle\sup_{1\leq k\leq M} (\ln \theta_{N_k,N_{k+1}})^{1/2}, & \mbox{ }q=\infty, \end{matrix}\right.\\
&+&C_{\epsilon,p}\lambda(N)\displaystyle\sum_{k=M+1}^{\infty}e^{-k}\theta_{N_k,N_{k+1}}^{1/p-1/q}.
\end{eqnarray}
Furthermore, by (\ref{eq:(4.2.51)}) and  (\ref{eq:(4.2.52)}) we have that 
\begin{eqnarray*}
 \lambda(N)\displaystyle\sum_{k=M+1}^{\infty}e^{-k}\theta_{N_k,N_{k+1}}^{1/p-1/q}&=& \lambda(N)\theta_{N_1,N_{2}}^{1/p-1/q}\displaystyle\sum_{k=M+1}^{\infty}e^{-k}\left(\frac{\theta_{N_k,N_{k+1}}}{\theta_{N_1,N_{2}}}\right)^{1/p-1/q}\\
 &\leq&C_4 \lambda(N)N^{(2d-r-1)(1/p-1/q)}\displaystyle\sum_{k=M+1}^{\infty}e^{-k}e^{\delta''k}\\
 &=&C_4 \lambda(N)N^{(2d-r-1)(1/p-1/q)}\left( e^{-(1-\delta'')}\right)^{M+1}\frac{1}{1-e^{-(1-\delta'')}}\\
 &\leq&C_6 \lambda(N)N^{(2d-r-1)(1/p-1/q)}\left( e^{-(1-\delta'')}\right)^{M},
\end{eqnarray*}
where $\delta''=\delta'(1/p-1/q)$, $\delta''<1$. We have that $M\asymp \epsilon^{-1}\ln \theta_{N_1,N_{2}}$ and $\theta_{N_1,N_{2}}\asymp N^{2d-r-1}$, and hence
\begin{eqnarray*}
 \lambda(N)\displaystyle\sum_{k=M+1}^{\infty}e^{-k}\theta_{N_k,N_{k+1}}^{1/p-1/q}&\leq&C_6 \lambda(N)N^{(2d-r-1)(1/p-1/q)}\left( e^{-(1-\delta'')}\right)^{C_5(\ln \theta_{N_1,N_{2}})/\epsilon}
 \leq C_7 \lambda(N),
\end{eqnarray*}
if $0<\epsilon<C_5(1-\delta'')/(1/p-1/q)$. Then it follows by (\ref{eq:(4.2.51)}) and (\ref{eq:(4.2.54)}) that
\begin{equation}\label{eq:(4.2.55)}
 d_\beta(\Lambda^{(2)}U_p,L^q)\ll e^{-\gamma N^r}N^{(2d-r-1)(1/p-1/2)}\left \{ \begin{matrix} q^{1/2}, & \mbox{ }2\leq q<\infty,
\\\displaystyle\sup_{1\leq k\leq M} (\ln \theta_{N_k,N_{k+1}})^{1/2}, & \mbox{ }q=\infty. \end{matrix}\right.
\end{equation}
By  (\ref{eq:(4.2.50)}) and (\ref{eq:(4.2.51)}), we get $\theta_{N_k,N_{k+1}}\asymp N_k^{2d-r-1}=\left(N_k^r\right)^{(2d-r-1)/r}\leq (N^r+C_1(k-1))^{(2d-r-1)/r}$ and considering that $M\leq \epsilon^{-1}\ln \theta_{N_1,N_{2}}$, $\theta_{N_1,N_{2}}\asymp N^{2d-r-1}$ and $N^{2d-1}\asymp s$ we get $\ln \theta_{N_k,N_{k+1}}\asymp \ln s$. Therefore by (\ref{eq:(4.2.55)})
\begin{equation*}\label{eq:(4.2.56)}
 d_\beta(\Lambda^{(2)}U_p,L^q)\ll e^{-\gamma N^r}N^{(2d-r-1)(1/p-1/2)}\left \{ \begin{matrix} q^{1/2}, & \mbox{ }2\leq q<\infty,
\\ (\ln s)^{1/2}, & \mbox{ }q=\infty. \end{matrix}\right.
\end{equation*}
From Remark \ref{rem4.1},  (\ref{eq:(89)}) and (\ref{eq:(4.2.51)}), we get 
\begin{equation*}
 \beta=m_0+\displaystyle\sum_{k=1}^M m_k=\displaystyle\sum_{s=0}^M \dim \mathcal{H}_s +\displaystyle\sum_{k=1}^M m_k
 \leq s + C_\epsilon\theta_{N_1,N_2}\leq s +C_8N^{2d-r-1},
\end{equation*}
and since $N^{2d-1}\asymp s$, it follows that $\beta\leq s +C_9s^{1-r/(2d-1)}$ and then
\begin{equation}\label{eq:(4.2.57)}
 d_{[ s +C_9s^{1-r/(2d-1)}]}(\Lambda^{(2)}U_p,L^q)\ll e^{-\gamma N^r}N^{(2d-r-1)(1/p-1/2)}\left \{ \begin{matrix} q^{1/2}, & \mbox{ }2\leq q<\infty,
\\ (\ln s)^{1/2}, & \mbox{ }q=\infty. \end{matrix}\right.
\end{equation}
Let $\mathcal{T}_N=s +C_9s^{1-r/(2d-1)}$. By (\ref{eq:(4.2.12)}) we have $-\gamma N^r\leq -\mathcal{R}s^{r/(2d-1)}+C_{10}s^{(r-1)/(2d-1)}$ and hence $C_{10}s^{(r-1)/(2d-1)}<C_{11}$. Thus
\begin{eqnarray*}
 -\gamma N^r+\mathcal{R}\tau_N^{r/(2d-1)}&\leq&-\mathcal{R}s^{r/(2d-1)}+\mathcal{R}\tau_N^{r/(2d-1)}+C_{11}\\
  &=&-\mathcal{R}\phi_N^{r/(2d-1)}+\mathcal{R}\phi_N^{r/(2d-1)}\left(1+C_9\phi_N^{-r/(2d-1)}\right)^{r/(2d-1)}+C_{11}\\
 &=&\mathcal{R}\phi_N^{r/(2d-1)}\left(-1+\left(1+C_9\phi_N^{-r/(2d-1)}\right)^{r/(2d-1)}\right)+C_{11} .
\end{eqnarray*}
Taking $N$ sufficiently large, we have $|C_9\phi_N^{-r/(2d-1)}|<1$ and hence we obtain
\begin{eqnarray*}
 & &\left(1+C_9\phi_N^{-r/(2d-1)}\right)^{r/(2d-1)}\\
 &=&1+\displaystyle\sum_{k=1}^\infty\frac{\left(r/(2d-1)\right)\left(r/(2d-1)-1\right)\cdots\left(r/(2d-1)-k+1\right)}{k!}\left(C_9\phi_N^{-r/(2d-1)}\right)^k\\
 &=&1+\frac{C_9r}{2d-1}\phi_N^{-r/(2d-1)}S_N,
\end{eqnarray*}
where $0\leq S_N\leq C_r$ and therefore $-\gamma N^r+\mathcal{R}\tau_N^{r/(2d-1)}\leq C_{12}$. 
Let $l\in\N$, $[\tau_N]\leq l\leq[\tau_{N+1}]$. We have $ 1<e^{-\gamma N^r}/e^{-\gamma (N+1)^r}\leq e^{\gamma r}$, so, using  (\ref{eq:(4.2.57)}), we obtain
\begin{eqnarray*}
  d_{l}(\Lambda^{(2)}U_p,L^q)&\leq& d_{[\tau_N]}(\Lambda^{(2)}U_p,L^q)\\
  &\leq& e^{-\gamma N^r}N^{(2d-r-1)(1/p-1/2)}\left \{ \begin{matrix} 1, & \mbox{ }2\leq q<\infty,
\\ (\ln \phi_N)^{1/2}, & \mbox{ }q=\infty, \end{matrix}\right.\\
&\ll& e^{-\gamma (N+1)^r}\left(N^{2d-1}\right)^{(2d-r-1)(1/p-1/2)/(2d-1)}\left \{ \begin{matrix} 1, & \mbox{ }2\leq q<\infty,
\\ (\ln \phi_N)^{1/2}, & \mbox{ }q=\infty, \end{matrix}\right.\\
&\ll& e^{-\mathcal{R}l^r/(2d-1)}l^{(1-r/(2d-1))(1/p-1/2)}\left \{ \begin{matrix} 1, & \mbox{ }2\leq q<\infty,
\\ (\ln l)^{1/2}, & \mbox{ }q=\infty, \end{matrix}\right.,
\end{eqnarray*}
proving (\ref{eq:(2.0.0000)}) for $1\leq p\leq 2\leq q\leq \infty$. 
The case $2\leq p\leq\infty$  follows because $\Lambda U_p \subset\Lambda U_2 $.\\
\indent Now for $r>1$ we will apply the Theorem \ref{2}. We have $-\gamma(k+1)^r+\gamma k^r\leq-\gamma rk^{r-1}$ for $k\geq 1$ and therefore $e\lambda(k+1)\leq e^{-\gamma rk^{r-1}+1}\lambda(k)$.
Let $a\in\N$ such that $e^{-\gamma rk^{r-1}+1}\leq 1$, for all $k\geq a$. Consider $N\geq a$, $N_0=0$, $N_1=N$, $N_{k+1}=N+k$, $M=0$, $\beta=m_0=n=\phi_N$.
Applying Theorem \ref{2} for $1\leq p\leq 2 \leq q\leq \infty$ we obtain
\begin{equation*}
 d_{\phi_N}(\Lambda^{(2)}U_p,L^q)\ll\displaystyle\sum_{k=1}^\infty \lambda(N_k)\theta_{N_k,N_{k+1}}^{1/p-1/q}=\displaystyle\sum_{k=1}^\infty \lambda(N+k-1)\theta_{N_k,N_{k+1}}^{1/p-1/q}.
\end{equation*}
Since $\lambda(N+k)\leq e^{-k}\lambda(N)$, for $1\leq p\leq2\leq q\leq\infty$ we obtain
\begin{eqnarray*}
 d_{\phi_N}(\Lambda^{(2)}U_p,L^q)&\ll&\lambda(N)\displaystyle\sum_{k=1}^\infty e^{-(k-1)}\theta_{N_k,N_{k+1}}^{1/p-1/q}\ll \lambda(N)\displaystyle\sum_{k=1}^\infty e^{-k}\theta_{N_k,N_{k+1}}^{1/p-1/q}\\
 &\ll& \lambda(N)\displaystyle\sum_{k=1}^\infty e^{-k}\left(\dim \mathcal{H}_{N+k}\right)^{1/p-1/q}\\
 &\ll& \lambda(N)\displaystyle\sum_{k=1}^\infty e^{-k}\left((N+k)^{2d-2}\right)^{1/p-1/q}\\
 &\ll& \lambda(N)\displaystyle\sum_{k=1}^\infty e^{-k}(Nk)^{(2d-2)(1/p-1/q)}\\
 &\ll& \lambda(N)N^{(2d-2)(1/p-1/q)}=e^{-\gamma N^r}N^{(2d-2)(1/p-1/q)}. 
\end{eqnarray*}
Now for $2\leq p,q\leq\infty$, we obtain
\begin{equation*}
 d_{\phi_N}(\Lambda^{(2)}U_p,L^q)\leq d_{\phi_N}(\Lambda^{(2)}U_2,L^q)\ll e^{-\gamma N^r}N^{(2d-2)(1/2-1/q)},
\end{equation*}
and this proves (\ref{eq:(2.0.00)}).
\end{proof}

	%
	%
	%
	%
	%
	
\end{document}